\documentclass{article}

\usepackage{color}
 \catcode`@=11
\usepackage{color}
\usepackage{amsthm,amssymb,amsfonts,mathrsfs}
\usepackage{amsmath}
\catcode`@=11 \@addtoreset{equation}{section}

\catcode`@=12

\newtheorem{theorem}{Theorem}[section]
\newtheorem{proposition}{Proposition}[section]
\newtheorem{lemma}{Lemma}[section]

\theoremstyle{definition}

\newtheorem{definition}{Definition}[section]

\usepackage[colorlinks=true,urlcolor=red,linkcolor=blue,citecolor=red,bookmarks=red]{hyperref}
\title{Boundary controllability for a coupled system of degenerate/singular parabolic equations}
\author{
{\sc  Brahim Allal}\\
Facult\'e des Sciences et Techniques,\\
Universit\'e Hassan 1er, Laboratoire MISI,\\
B.P. 577, Settat 26000, Morocco\\
email: b.allal@uhp.ac.ma\\\\
{\sc  Abdelkarim Hajjaj}\\
Facult\'e des Sciences et Techniques,\\
Universit\'e Hassan 1er, Laboratoire MISI,\\
B.P. 577, Settat 26000, Morocco\\
email: abdelkarim.hajjaj@uhp.ac.ma\\\\
{\sc  Jawad Salhi}\\
Moulay Ismail University of Meknes\\
FST Errachidia, MAIS Laboratory, MAMCS Group,\\
P.O. Box 509, Boutalamine 52000, Errachidia, Morocco\\
email: sj.salhi@gmail.com\\\\
{\sc Amine Sbai }\\
Facult\'e des Sciences et Techniques,\\
Universit\'e Hassan 1er, Laboratoire MISI,\\
B.P. 577, Settat 26000, Morocco\\
email: a.sbai@uhp.ac.ma
}

\date{}
\begin{document}

\maketitle

\begin{abstract}
In this paper we study the boundary controllability for a system of two coupled degenerate/singular parabolic equations with a control acting on only one equation. We analyze both approximate and null boundary controllability properties. 
Besides, we provide an estimate on the null-control cost. The proofs are based on the use of the moment method by Fattorini and Russell together with some results on biorthogonal families.
\end{abstract}


\section{Introduction}
This work is devoted to the study of the boundary controllability properties of the following controlled system:
\begin{equation}\label{bound-sys}
\left\{
\begin{array}{lll}
y_t - (x^\alpha y_x)_x - \frac{\mu}{x^{2-\alpha}}y = A y ,  &  & (t,x)\in Q := (0, T ) \times (0, 1), \\
y(t,1)= B v, & &  t \in (0, T),  \\
y(t, 0) = 0,   & & t \in (0, T),\\
y(0,x)= y_{0}(x),  & & x\in (0,1),
\end{array}
\right.
\end{equation}
where $T > 0$  is some final time, $0 \leq \alpha < 1$, $\mu \leq \mu(\alpha) = \frac{(1 - \alpha)^2}{4}$, $y_0 \in L^2(0, 1)$, $y=(y_1,y_2)^*$ is the state variable and $v=v(t)$ is the control function which acts on the system by means of the Dirichlet boundary condition at the point $x = 1$. Moreover, $A \in \mathcal{L}(\mathbb{R}^2)$ and $B \in \mathbb{R}^2$ are, respectively, a suitable coupling matrix and a control operator, chosen so that:
\begin{equation}\label{mainhyp}
rank [B | AB] = 2.
\end{equation}
Notice that, taking $P = [B|AB]$, by performing the
change of variables $\tilde{y} = P^{-1}y$, one obtains the following reformulation
of \eqref{bound-sys}:
\begin{equation}\label{new-sys}
\left\{
\begin{array}{lll}
\tilde{y}_t - (x^\alpha \tilde{y}_x)_x - \frac{\mu}{x^{2-\alpha}}\tilde{y} = \tilde{A} \tilde{y},  &  & (t,x)\in (0, T) \times (0, 1), \\
\tilde{y}(t,1)= \tilde{B} v, & & t\in (0,T),  \\
\tilde{y}(t, 0) = 0,  & & t \in (0,T),\\
\tilde{y}(0,x)= P^{-1} y_{0}(x),  & & x\in(0, 1),
\end{array}
\right.
\end{equation}
where
\begin{align*}
\tilde{A} = \begin{pmatrix} 0 & a_1 \\1 & a_2 \end{pmatrix}\quad \text{and}\quad \tilde{B}=e_1=\left(\begin{matrix}
 1 \\
0
\end{matrix}\right).
\end{align*}
Therefore, for simplicity, it will be assumed in the rest of the paper that $A$ and $B$ are given by
\begin{align}\label{AandB}
A= \begin{pmatrix} 0 & a_1 \\1 & a_2 \end{pmatrix}
\quad \text{and}\quad
B=e_1
=\left(\begin{matrix}
 1 \\
 0
\end{matrix}\right).
\end{align}
This means that we are exerting only one control force on the system but we want to control the corresponding state $y=(y_1,y_2)$ which has two components.

The starting point of the present work is the results established in \cite{FBT} for the boundary controllability properties of the (uniformly) parabolic system:
\begin{equation}\label{bound-sys-cara}
\left\{
\begin{array}{lll}
y_t -  y_{xx} = A y ,  &  & (t,x)\in Q, \\
y(t, 0) = B v,   & & t \in (0, T),\\
y(t,1)= 0, & &  t \in (0, T),  \\
y(0,x)= y_{0}(x),  & & x\in (0,1),
\end{array}
\right.
\end{equation}
with $A$ and $B$ as defined previously. In fact, as shown in \cite{FBT}, there exists two different situations:
\begin{enumerate}
\item If the matrix $A$ in \eqref{bound-sys-cara} has one
double real eigenvalue or a couple of conjugate complex eigenvalues, \eqref{mainhyp} is a necessary and
sufficient condition for the null controllability at any time.
\item If $A$ has two different real eigenvalues, an additional condition is needed for null controllability,
independently of the considered vector $B$. 
\end{enumerate}

Since then, several other works followed extending them in various situations. See for instance  \cite{KBGT2016,KBGT2014,ABGT2011,ben2020,BBGBO14,duprez,BandN}.

However, all the previous cited papers deal with uniformly parabolic problems without degeneracies or singularities.
In recent years, controllability issues for degenerate and/or singular parabolic problems by means of a locally distributed control have been investigated in many papers, see \cite{BZ2016,Cazacu2014,Ervedoza2008,F2016,FM2017,Van2011,VanZua2008}. For related systems of coupled degenerate/singular parabolic equations we refer to \cite{AHMS2020,HMS2016,S2018}.

To our best knowledge, the known boundary controllability results of degenerate and/or singular parabolic problems are in the
scalar case (see for example \cite{biccari2019,biccsantavan2020,CMV2017,CMV2020,Gueye,MarVan2019}).

The aim of this research is to establish general results in the case of coupled degenerate/singular parabolic equations.
To prove our results we will use the moment method by Fattorini and Russell, introduced in \cite{fatrus1974,fatrus1971} in the framework of the boundary controllability of the one-dimensional scalar heat equation.

At first, we will see that for every $v\in L^2(0,T)$ and $y_0\in H^{-1,\mu}_{\alpha}(0,1)^2$, system \eqref{bound-sys} admits a unique weak solution defined by transposition that satisfies
\begin{equation*}
y \in  L^2(Q)^2  \cap  C^0\big([0, T], H^{-1,\mu}_{\alpha}(0, 1)^2 \big).
\end{equation*}
Observe that the previous regularity permits to pose the boundary controllability of the singular system \eqref{bound-sys} in the space $H^{-1,\mu}_{\alpha}(0, 1)^2$ (that will be defined later in section \ref{Section-prel}). 

Then, we pass to analyse both approximate and  null controllability issues using a boundary
control acting at $x = 1$. So, we use the following notions:
\begin{definition}
\begin{enumerate}
\item It will be said that system \eqref{bound-sys} is approximately controllable in $H^{-1,\mu}_{\alpha}(0, 1)^2$ at time $T>0$ if for every $y_0, y_d \in H^{-1,\mu}_{\alpha}(0,1)^2$ and any $\varepsilon > 0$, there exists a control function $v\in L^2(0,T)$ such that the solution $y$ to system \eqref{bound-sys} satisfies
\begin{equation*}
\|y(T,\cdot)-y_d\|_{H^{-1,\mu}_{\alpha}(0,1)^2}\leq \varepsilon.
\end{equation*}
\item It will be said that system \eqref{bound-sys} is null controllable at time $T>0$ if for every $y_0 \in H^{-1,\mu}_{\alpha}(0, 1)^2$, there exists a control $v\in L^2(0,T)$ such that the solution $y$ to system \eqref{bound-sys} satisfies
\begin{equation*}
y(T,\cdot)=0,\quad \text{in}\quad H^{-1,\mu}_{\alpha}(0, 1)^2.
\end{equation*}
\end{enumerate}
\end{definition}
We emphasize that imposing a control that acts at the nonsingular point does not imply a simple adaptation of the previous distributed controllability results. For example, at a first glance, one may think that our boundary controllability results can be obtained  directly by standard extension and localization arguments from the corresponding distributed controllability results as in the case of scalar  parabolic equations. But this is not the case and the situation is quite different for non-scalar parabolic systems. Indeed, while the Kalman's rank condition \eqref{mainhyp} is a necessary and sufficient condition for the null controllability at any time in the distributed case, it was proved in \cite{FBT} that it is necessary, but not sufficient, for the boundary controllability for coupled parabolic systems.

In order to get the approximate controllability result of \eqref{bound-sys}, we will need the following known result (see \cite{ABGT2011} or \cite{FBT}).

\begin{theorem}\label{thm-biorth0}
Let $T>0$. Suppose that $ \{\Lambda_n\}_{n\geq1} $ is a sequence of complex numbers such that, for some $ \delta, \rho >0$, one has
\begin{equation}\label{Cv-gap-orthog0}
\left\{
\begin{array}{l}
\Re(\Lambda_n) \geq \delta |\Lambda_n|, \qquad |\Lambda_n - \Lambda_m| \geq \rho |n - m|, \quad \forall n, m \geq 1, \\
\sum\limits_{n\geq 1} \dfrac{1}{|\Lambda_n|} < + \infty .
\end{array}
\right.
\end{equation}
Then, there exists a family $\{q_n \}_{n\geq 1} \subset L^2(0, T) $ biorthogonal to $\{e^{- \Lambda_n t}\}_{n\geq 1}$ i.e., a family $\{q_n \}_{n\geq 1} $ in $L^2(0, T)$ such that
$$ \int_0^T q_n(t) e^{- \Lambda_m t} \, dt = \delta_{nm}, \quad \forall n,m\geq 1.$$
Moreover, for every $ \varepsilon > 0 $, there exists $ C_{\varepsilon} > 0 $ for which
\begin{equation*}
\|q_n \|_{L^2(0, T)} \leq C_\varepsilon e^{\varepsilon \Re(\Lambda_n)}, \qquad \forall n \geq 1.
\end{equation*}
\end{theorem}

It is worth mentioning that the above Theorem can also be applied to get the null controllability result for the system \eqref{bound-sys}. However, it does not permit to deduce the required exponential estimate on the null-control cost.

For this reason, to obtain the null controllability result together with an estimate of the control cost, we are going to apply the next result provided in \cite{BBGBO14}.
\begin{theorem}\label{thm-biorth}
Let $ \{\Lambda_n\}_{n\geq1} $ be a sequence of complex numbers fulfilling the following assumptions:
\begin{enumerate}
\item
$\Lambda_n \neq \Lambda_m$ for all $n, m \geq 1$ with $n \neq m$;
\item
$\Re (\Lambda_n) > 0$ for every $n \geq 1$;
\item
for some $ \delta >0$
$$|\Im (\Lambda_n)| \leq \delta \sqrt{\Re (\Lambda_n)} \quad \forall n \geq 1;$$
\item
$ \{\Lambda_n\}_{n\geq1} $ is nondecreasing in modulus,
$$ |\Lambda_n| \leq |\Lambda_{n+1}| \quad \forall n \geq 1; $$
\item
$ \{\Lambda_n\}_{n\geq1} $ satisfies the following gap condition: for some
$\varrho, q >0 $,
\begin{equation}\label{strong-gap}
\left\{
\begin{array}{l}
|\Lambda_n - \Lambda_m| \geq \varrho |n^2 - m^2| \quad \forall n, m: |n -m|\geq q, \\
\inf\limits_{n\neq m, \; |n -m|<q}|\Lambda_n - \Lambda_m| > 0;
\end{array}
\right.
\end{equation}
\item
for some $p, s> 0,$
\begin{equation}\label{ineq-counting}
| p \sqrt{r} - \mathcal{N}(r)| \leq s, \quad \forall r >0,
\end{equation}
where $\mathcal{N}$ is the counting function associated with the sequence
$ \{\Lambda_n\}_{n\geq1} $, that is the function defined by
\begin{equation*}
\mathcal{N}(r) = \#\{ n : \, |\Lambda_n| \leq r \}, \qquad \forall r>0.
\end{equation*}
\end{enumerate}
Then, there exists $T_0 > 0 $, such that for any $T\in (0,T_0)$, we can find a family $\{q_n \}_{n\geq 1} \subset L^2(-T/2, T/2) $ biorthogonal to $\{e^{- \Lambda_n t}\}_{n\geq 1}$ i.e., a family $\{q_n \}_{n\geq 1} $ in $L^2(-T/2, T/2)$ such that
$$ \int_{-T/2}^{T/2} q_n(t) e^{- \Lambda_m t} \, dt = \delta_{nm}. $$
Moreover, there exists a positive constant $ C > 0 $  independent of $T$ for which
\begin{equation}\label{boundqni}
\|q_n \|_{L^2(-T/2, T/2)} \leq C e^{C\sqrt{\Re (\Lambda_n)} + \frac{C}{T}}, \qquad \forall n \geq 1.
\end{equation}
Here for $z\in \mathbb{C}$, $\Re (z)$  and $\Im (z)$ denote the real and imaginary parts of $z$.
\end{theorem}

The rest of the paper is organized as follows. In Section \ref{Section-prel}, we prove the well-posedness of the problem \eqref{bound-sys} in appropriate weighted spaces using the transposition method and recall
some characterizations of the controllability.
In section \ref{Section-eigenvalue}, we discuss the spectral analysis related to scalar singular operators and present a description of
the spectrum associated with system \eqref{bound-sys} which will be useful for developing the moment method.
Section \ref{Section-approx} is devoted to studying the boundary approximate controllability problem for the system \eqref{bound-sys}.
Finally, in section \ref{Section-null}, we prove the boundary null controllability result and establish an estimate of the control cost.

\section{Preliminary results}\label{Section-prel}
\subsection{Functional framework}
\label{CP}
In the study of degenerate/singular problems, it is by now classical
that of great importance is the following generalized Hardy inequality (see, for example, \cite{Van2011} or \cite[Lemma 5.3.1]{Davies1995}): for all $\alpha\in [0,2)$,
\begin{equation}\label{hardyinequality}
\frac{(1-\alpha)^2}{4}\int_{0}^{1}x^{\alpha-2}z^2\,dx \leq \int_{0}^{1}x^{\alpha}z^2_x\,dx,\quad \forall z\in C^{\infty}_{c}(0,1).
\end{equation}

For any $\mu\leq \mu(\alpha)$, we introduce the functional space associated to degenerate/singular problems:
\begin{align*}
&H_{\alpha}^{1,\mu}(0,1):= \Big\{ z \in L^2(0,1)\cap H^{1}_{loc}((0,1])\,\mid \int_{0}^{1}(x^{\alpha}z_x^2 - \mu \frac{z^2}{x^{2-\alpha}})\,dx<+\infty \Big\}\\
&H_{\alpha,0}^{1,\mu}(0,1):= \Big\{ z \in H_{\alpha}^{1,\mu}(0,1)\,\mid z(0)=z(1)=0 \Big\}.
\end{align*}
Further, we define $H^{-1,\mu}_{\alpha,0}(0,1)$  the dual space of $H_{\alpha,0}^{1,\mu}(0,1)$ with respect to the pivot space
$L^2(0,1)$, endowed with the natural norm
\begin{equation*}
\| f \|_{H_{\alpha,0}^{-1,\mu}(0,1)}:=  \sup_{\| g \|_{H_{\alpha,0}^{1,\mu}(0,1)}=1} \langle f, g \rangle_{H^{-1,\mu}_{\alpha,0}(0,1), H_{\alpha,0}^{1,\mu}(0,1)}.
\end{equation*}
We also define
\begin{equation*}
H^{2,\mu}_{\alpha}(0,1)= \big\{z\in H^{1,\mu}_{\alpha}(0,1)\cap H_{loc}^2((0,1])\quad \mid \quad (x^{\alpha}z_{x})_x + \frac{\mu}{x^{2-\alpha}}z \in L^2(0,1) \big\}.
\end{equation*}

Notice besides that, as $C^{\infty}_{c}(0,1)$ is dense both in $L^2(0,1)$ and in $H_{\alpha,0}^{1,\mu}(0,1)$, $H_{\alpha,0}^{1,\mu}(0,1)$ is dense in $L^2(0,1)$.

In what follows, for simplicity, we will always denote
by $\langle \cdot,\cdot \rangle$ the standard scalar product of either $L^2(0,1)$ or $L^2(0,1)^2$, by $\langle \cdot,\cdot \rangle_{X',X}$ the duality
pairing between the Hilbert space $X$ and its dual $X'$.  On the other hand, we will use $\|\cdot \|^{\mu}_{\alpha}$ (resp. $\|\cdot \|^{-1,\mu}_{\alpha}$) for denoting the norm of $H_{\alpha,0}^{1,\mu}(0,1)^2$ (resp. $H^{-1,\mu}_{\alpha}(0,1)^2$).

\subsection{Well-posedness}

Now, we are ready to give some results related to the existence, uniqueness and continuous
dependence with respect to the data of the degenerate/singular problem \eqref{bound-sys}.
To this aim, let us consider the following nonhomogeneous adjoint problem:
\begin{equation}\label{nhadjoint-sys}
\left\{
\begin{array}{lll}
- \varphi_t - (x^{\alpha}\varphi_{x})_x - \frac{\mu}{x^{2-\alpha}}\varphi = A^* \varphi + g,  &  & \text{in} \; Q, \\
\varphi(t, 0)=\varphi(t,1) = 0, & & t \in (0, T),\\
\varphi(T,x)= \varphi_0,  & & \text{in}\; (0, 1),
\end{array}
\right.
\end{equation}
where $A$ is given in \eqref{AandB} and $\varphi_0$ and $g$ are functions in appropriate spaces.

Let us start with a first result on existence and uniqueness of strict solutions to system \eqref{nhadjoint-sys}.
One has (see \cite{allal2021} or \cite[Definition 4.1]{biccsantavan2020}):
\begin{proposition}\label{well-posed}
Assume that $\varphi_0\in H_{\alpha,0}^{1,\mu}(0,1)^2$ and $g\in L^2(Q)^2$. Then, system \eqref{nhadjoint-sys} admits a unique strict solution
\begin{equation*}
\begin{aligned}
\varphi \in  C^0([0, T]; H_{\alpha,0}^{1,\mu}(0,1)^2)&\cap H^1(0,T;L^2(0,1)^2)\\
&\cap L^2(0, T; H_{\alpha}^{2,\mu}(0,1)^2\cap H_{\alpha,0}^{1,\mu}(0,1)^2)
\end{aligned}
\end{equation*}
such that
\begin{equation}\label{energy-sol}
\begin{aligned}
\|\varphi\|_{C^0([0, T]; H_{\alpha,0}^{1,\mu}(0,1)^2)} &+ \|\varphi\|_{H^1(0,T;L^2(0,1)^2)} + \|\varphi\|_{L^2(0, T;  H_{\alpha}^{2,\mu}(0,1)^2\cap H_{\alpha,0}^{1,\mu}(0,1)^2)}  \\
&\leq C \Big(\|\varphi_0\|^{\mu}_{\alpha} + \|g\|_{L^2(Q)^2} \Big),
\end{aligned}
\end{equation}
for some positive constant $C$.
\end{proposition}
In view of Proposition \ref{well-posed}, the following definition makes sense:
\begin{definition}\label{transp}
Let $y_0 \in H_{\alpha}^{-1,\mu}(0,1)^2$ and $v \in L^2(0, T)$ be given. It will be said that $y \in L^2(Q)^2$ is a solution by transposition to \eqref{bound-sys} if, for each $g \in L^2(Q)^2$, the following identity holds
\begin{equation}\label{def-transp}
\int\!\!\!\!\!\int_{Q} y \cdot g \,dx\,dt  = \langle y_0 , \varphi(0, \cdot) \rangle_{H_{\alpha}^{-1,\mu}, H_{\alpha,0}^{1,\mu}} - \int_0^T B^* \varphi_x(t,1)\,v(t)\,dt,
\end{equation}
where
$\varphi \in  C^0([0, T]; H_{\alpha,0}^{1,\mu}(0,1)^2)\cap H^1(0,T;L^2(0,1)^2)\cap L^2(0, T; H_{\alpha}^{2,\mu}(0,1)^2\cap H_{\alpha,0}^{1,\mu}(0,1)^2)$
is the solution of \eqref{nhadjoint-sys} associated
to $g$ and $\varphi_0=0$.
\end{definition}
With this definition we can state the result of existence and uniqueness of solution by transposition to system \eqref{bound-sys}.
\begin{proposition}\label{prop-transp}
Assume that $y_0 \in H_{\alpha}^{-1,\mu}(0, 1)^2$ and $v \in L^2(0, T)$. Then, system \eqref{bound-sys} admits a unique solution by transposition $y$ that satisfies
\begin{equation}\label{reg-faible}
\left\{
\begin{array}{lll}
y \in  L^2(Q)^2  \cap  C^0\big([0, T], H_{\alpha}^{-1,\mu}(0, 1)^2 \big),\\
y_t  \in L^2\big(0, T; (H_{\alpha}^{2,\mu}(0,1)^2\cap H_{\alpha,0}^{1,\mu}(0,1)^2)^{'}),  \\
y_t - (x^{\alpha}y_{x})_x - \frac{\mu}{x^{2-\alpha}}y = A y  \quad \text{in} \quad L^2\big(0, T; (H_{\alpha}^{2,\mu}(0,1)^2\cap H_{\alpha,0}^{1,\mu}(0,1)^2)^{'}) , \\
y(0,\cdot)= y_0 \quad \text{in} \quad  H_{\alpha}^{-1,\mu}(0, 1)^2
\end{array}
\right.
\end{equation}
and
\begin{equation}\label{energy-transp}
\begin{aligned}
&\|y \|_{L^{2}(Q)^2} + \|y\|_{C^0(H_{\alpha}^{-1,\mu})} + \|y_t\|_{ L^2((H_{\alpha}^{2,\mu}(0,1)^2\cap H_{\alpha,0}^{1,\mu}(0,1)^2)^{'})} \\
&\leq C \big(\|v \|_{L^2(0,T)} + \|y_0\|_{\alpha}^{-1,\mu} \big),
\end{aligned}
\end{equation}
for a constant $C=C(T)> 0$.
\end{proposition}

\begin{proof}
\smallskip
Let $y_0 \in H_{\alpha}^{-1,\mu}(0, 1)^2 $, $ v \in L^2(0, T)$ and consider the following functional $ \mathcal{T}: L^2(Q)^2 \rightarrow \mathbb{R}$
given by
\begin{equation*}
\mathcal{T}(g) = \langle y_0 , \varphi(0, \cdot) \rangle_{H_{\alpha}^{-1,\mu}, H_{\alpha,0}^{1,\mu}} - \int_0^T B^* \varphi_x(t,1) v(t)\;dt,
\end{equation*}
where
$
\varphi \in  C^0([0, T]; H_{\alpha,0}^{1,\mu}(0,1)^2)\cap H^1(0,T;L^2(0,1)^2)\cap L^2(0, T; H_{\alpha}^{2,\mu}(0,1)^2\cap H_{\alpha,0}^{1,\mu}(0,1)^2)
$
is the solution of the adjoint system \eqref{nhadjoint-sys} associated to $g \in L^2(Q)^2$ and $\varphi_0=0$. From \eqref{energy-sol},
we can deduce the existence of a positive constant $C$ such that
\begin{equation*}
\big|\mathcal{T}(g) \big| \leq C \big( \| v \|_{L^2(0,T)} + \|y_0\|_{\alpha}^{-1,\mu}  \big) \|g\|_{L^2(Q)^2},
\end{equation*}
for all $g \in L^2(Q)^2$.
We infer that $\mathcal{T}$ is bounded.
Hence, by Riesz-Fr\'echet representation theorem, there exists a unique $ y \in L^2(Q)^2$ satisfying
\eqref{def-transp}, i.e., a solution by transposition of \eqref{bound-sys} in the sense of Definition \ref{transp}.
It is also clear that this solution satisfies the equality
$y_t - (x^{\alpha}y_{x})_x -\frac{\mu}{x^{2-\alpha}} y = A y$ in $\mathcal{D}^{'}(Q)^2$ and the estimate
\begin{equation*}
\| y \|_{L^2(Q)^2}= \|\mathcal{T}\| \leq C \big(\| v \|_{L^2(0,T)} +  \|y_0\|_{\alpha}^{-1,\mu} \big).
\end{equation*}
Next, we are going to prove that the solution $y$ of system \eqref{bound-sys} is more regular. To be precise, let us show that
$(x^{\alpha}y_{x})_x +\frac{\mu}{x^{2-\alpha}}y \in  L^2\big(0, T; (H_{\alpha}^{2,\mu}(0,1)^2\cap H_{\alpha,0}^{1,\mu}(0,1)^2)^{'})$ and
\begin{equation}\label{bounddiv}
\| (x^{\alpha}y_{x})_x +\frac{\mu}{x^{2-\alpha}}y \|_{L^2((H_{\alpha}^{2,\mu}(0,1)^2\cap H_{\alpha,0}^{1,\mu}(0,1)^2)^{'})} \leq C \big(\| v \|_{L^2(0,T)} + \|y_0\|_{\alpha}^{-1,\mu} \big).
\end{equation}
To this end, let us consider two sequences $\{y_0^m \}_{m \geq 1} \subset H_{\alpha,0}^{1,\mu}(0,1)^2$ and $\{v^m \}_{m\geq1} \subset H_{0}^{1}(0, T)$ such that
\begin{equation*}
y_0^m \rightarrow y_0 \quad \text{in} \quad H^{-1,\mu}(0,1)^2 \quad \text{and} \quad  v^m \rightarrow v \quad \text{in} \quad L^2(0,T).
\end{equation*}
Now, the strategy consists in transforming our original system \eqref{bound-sys} (as done for instance in \cite{biccsantavan2020} in
the context of a scalar degenerate/singular parabolic equation) into a problem with homogeneous boundary conditions and a source term.
To this end, let us introduce the following function:
\begin{equation*}
\forall x\in[0,1],\qquad p(x):=x^{q^\mu_\alpha} \qquad \text{where}\qquad q_\alpha^\mu:= \frac{1-\alpha}{2}+ \sqrt{\mu(\alpha)-\mu}.
\end{equation*}
Formally, if $y^m$ is the solution of \eqref{bound-sys} associated to $y_0^m$ and $v^m$, then the function defined by
\begin{equation*}
 \tilde{y}^m(t,x)= y^m(t,x) - B p(x) v^m(t),
\end{equation*}
is solution of
\begin{equation}\label{tilde-sys}
\left\{
\begin{array}{lll}
\tilde{y}_t^m - (x^{\alpha}\tilde{y}_{x}^m)_x - \frac{\mu}{x^{2-\alpha}}\tilde{y}^m  = A \tilde{y}^m + \tilde{f}^m(t,x), & &  (t,x) \in Q, \\
\tilde{y}^m(t, 0) =\tilde{y}^m(t,1)= 0, & & t \in (0,T),\\
\tilde{y}^m(0,x)= y_{0}^m(x),  & &  x\in (0, 1),
\end{array}
\right.
\end{equation}
where $\tilde{f}^m(t,x) =  p(x) v^m(t) A B - p(x) v_t^m(t) B \in L^2(Q)^2$.
With the previous regularity assumptions on the data, we can apply Proposition \ref{well-posed}, to deduce that system \eqref{tilde-sys} has a unique
strict solution
\begin{equation*}
\begin{aligned}
\tilde{y}^m \in  C^0([0, T]; H_{\alpha,0}^{1,\mu}(0,1)^2)&\cap H^1(0,T;L^2(0,1)^2)\\
&\cap L^2(0, T; H^{2,\mu}(0,1)^2\cap H_0^{1,\mu}(0,1)^2).
\end{aligned}
\end{equation*}
By setting
$$
\tilde{v}^m(t,x):= B p(x) v^m(t),
$$
we observe that $\tilde{v}^m$ satisfies
\begin{equation*}
\tilde{v}^m \in  C^0([0, T]; H_{\alpha}^{1,\mu}(0,1)^2)\cap H^1(0,T;L^2(0,1)^2)\cap L^2(0, T; H_{\alpha}^{2,\mu}(0,1)^2).
\end{equation*}
Therefore, the problem \eqref{bound-sys} for $v^m$ and  $y_0^m$ has a unique solution
$$y^m \in C^0([0, T]; H_{\alpha}^{1,\mu}(0,1)^2)\cap H^1(0,T;L^2(0,1)^2)\cap L^2(0, T; H_{\alpha}^{2,\mu}(0,1)^2)$$
which satisfies
\begin{equation*}
\int\!\!\!\!\!\int_{Q} y^m \cdot g \; dt dx  = \langle y_0^m , \varphi(0,x) \rangle_{H_{\alpha}^{-1,\mu}, H_{\alpha,0}^{1,\mu}} - \int_0^T B^*\varphi_x (t,1) v^m(t)\,dt, \quad \forall m \geq 1,
\end{equation*}
for all $ g \in L^2(Q)^2$, where $\varphi$ is the solution of the system \eqref{nhadjoint-sys} associated to $g$ and $\varphi_0=0$. The previous identity and \eqref{def-transp} also provide:
\begin{equation}\label{estym}
\left\{
\begin{array}{lll}
\|y^m\|_{L^2(Q)^2} \leq C \big(\| v \|_{L^2(0,T)} + \|y_0\|_{\alpha}^{-1,\mu} \big) \\
y^m \rightarrow y \quad \text{in} \; L^2(Q)^2 \quad \text{and} \quad  (x^{\alpha}y_{x}^m)_x +\frac{\mu}{x^{2-\alpha}}y^m  \rightarrow (x^{\alpha}y_{x})_x +\frac{\mu}{x^{2-\alpha}} y \quad \text{in} \; \mathcal{D}^{'}(Q)^2.
\end{array}
\right.
\end{equation}
On the other hand, integrations by parts lead to
\begin{equation*}
\int\!\!\!\!\!\int_{Q} \big((x^{\alpha}y_{x}^m)_x +\frac{\mu}{x^{2-\alpha}}y^m \big)  \cdot  \psi \,dt=
\int\!\!\!\!\!\int_{Q} y^m \cdot \big((x^{\alpha}\psi_{x})_x +\frac{\mu}{x^{2-\alpha}}\psi \big)\,dt\,dx - \int_0^T B^*\psi_x (t,1)\,v^m(t)\,dt,
\end{equation*}
for every $ \psi \in L^2\big(0, T; H_{\alpha}^{2,\mu}(0,1)^2\cap H_{\alpha,0}^{1,\mu}(0,1)^2)$. From this equality we deduce
that the sequence $ \{(x^{\alpha}y_{x}^m)_x +\frac{\mu}{x^{2-\alpha}}y^m \}_{m\geq1}$ is bounded in $L^2\big(0, T; (H_{\alpha}^{2,\mu}(0,1)^2\cap H_{\alpha,0}^{1,\mu}(0,1)^2)^{'})$.
This property together with \eqref{estym} implies that
$(x^{\alpha}y_{x})_x +\frac{\mu}{x^{2-\alpha}}y \in L^2\big(0, T; (H_{\alpha}^{2,\mu}(0,1)^2\cap H_{\alpha,0}^{1,\mu}(0,1)^2)^{'})$ and satisfies the estimate \eqref{bounddiv}.

Combining the identity $y_t = (x^{\alpha}y_{x})_x +\frac{\mu}{x^{2-\alpha}}y + A y$
and the regularity property for $(x^{\alpha}y_{x})_x +\frac{\mu}{x^{2-\alpha}}y$, we also see
that $y_t \in  L^2\big(0, T; (H_{\alpha}^{2,\mu}(0,1)^2\cap H_{\alpha,0}^{1,\mu}(0,1)^2)^{'})$ and
$$\|y_t \|_{L^2((H_{\alpha}^{2,\mu}(0,1)^2\cap H_{\alpha,0}^{1,\mu}(0,1)^2)^{'})} \leq C \big(\| v \|_{L^2(0,T)} + \|y_0\|_{\alpha}^{-1,\mu} \big),$$
for some positive constant $C$.
Therefore $ y \in C([0, T] ; X^2) $, where $X$ is the interpolation space
$X = [L^2(0,1),  (H_{\alpha}^{2,\mu}(0,1)^2\cap H_{\alpha,0}^{1,\mu}(0,1))^{'}]_{1/2} = H_{\alpha}^{-1,\mu}(0,1)$ (see \cite[Proposition 2.1, p. 22]{Lions68}).
In conclusion, we get
\begin{equation*}
\|y \|_{C(H_{\alpha}^{-1,\mu})} \leq C \big(\| v \|_{L^2(0,T)} + \|y_0 \|_{\alpha}^{-1,\mu} \big).
\end{equation*}
Finally, one can easily check that $ y(0,\cdot) =y_0$ in $H_{\alpha}^{-1,\mu}(0,1)^2$. This ends the proof.
\end{proof}

\subsection{Duality}

Let us conside the adjoint system of \eqref{bound-sys} given by:
\begin{equation}\label{adjoint-sys}
\left\{
\begin{array}{lll}
- \varphi_t - (x^\alpha\varphi_{x})_{x} - \frac{\mu}{x^{2-\alpha}}\varphi = A^* \varphi ,  &  & \text{in} \; Q, \\
\varphi(t, 0)=\varphi(t,1) = 0, & & t \in (0, T),\\
\varphi(T,x)= \varphi_0,  & & \text{in}\; (0, 1),
\end{array}
\right.
\end{equation}
where $\varphi_0\in H_{\alpha,0}^{1,\mu}(0,1)^2$.
In the sequel, the solution to \eqref{adjoint-sys} will be called the adjoint state associated
to $\varphi_0$.
The controllability of system \eqref{bound-sys} can be characterized in terms of appropriate properties of the
solutions to \eqref{adjoint-sys}.
In order to provide these characterizations, we use the following result
which relates the solutions of systems \eqref{bound-sys} and \eqref{adjoint-sys}. One has:
\begin{proposition}\label{mainrelation}
Let $y_0\in H^{-1,\mu}_{\alpha}(0,1)^2$, $v \in L^2(0,T)$ and $\varphi_0\in H_{\alpha,0}^{1,\mu}(0,1)^2$ be given.
Let $y$ be the state associated
to $y_0$ and $v$ and let $\varphi$ be the adjoint state associated to $\varphi_0$. Then:
\begin{equation}\label{relationship}
\int_0^T B^*(x^\alpha \varphi_x)(t, 1) v(t)\,dt = \langle y_0 , \varphi(0, \cdot) \rangle_{H^{-1,\mu}_{\alpha}, H_{\alpha,0}^{1,\mu}}
-  \langle y(T) , \varphi_0 \rangle_{H^{-1,\mu}_{\alpha}, H_{\alpha,0}^{1,\mu}}.
\end{equation}
\end{proposition}
This result is a straightforward consequence of the properties of $y$ stated in Proposition \ref{prop-transp}.

One important result that will be useful for treating the approximate controllability of the system \eqref{bound-sys} is the following characterization
in terms of the unique continuation property for the corresponding adjoint system \eqref{adjoint-sys}. More precisely, we have:
\begin{theorem}\label{approxnullcontrol}
Let us consider $T > 0$. Then,
system \eqref{bound-sys} is approximately controllable at time $T$ if and only if for all initial condition $\varphi_0\in H_{\alpha,0}^{1,\mu}(0,1)^2$
the solution to system \eqref{adjoint-sys} satisfies the unique continuation property
\begin{equation*}
B^* (x^\alpha \varphi_x)(\cdot,1) = 0  \quad \text{on} \quad  (0,T) \Rightarrow \varphi_0 = 0 \quad \text{in} \quad (0,1) \quad
(\text{i.e.},\quad \varphi = 0 \quad \text{in} \quad Q).
\end{equation*}
\end{theorem}
This result is well known. For a proof see, for instance \cite{FBT}, \cite{coron} and \cite{Zab1995}.

\section{Spectral analysis}\label{Section-eigenvalue}
In order to transform the question of null controllability
into a moment problem, we need to study the eigenvalue problem of the degenerate/singular operator associated to system \eqref{bound-sys}. To this end, we first recall the spectral properties of scalar degenerate/singular operators.
\subsection{Scalar degenerate/singular operators}
In this section, we discuss some preliminary results related to a spectral analysis of the operator $y \mapsto  - (x^\alpha y_x)_x - \frac{\mu}{x^{2-\alpha}}y$, i.e.,
the nontrivial solutions $(\lambda, \Phi)$ of
\begin{equation}\label{eigenvaluesproblem}
\begin{cases}
& -(x^\alpha \Phi')'(x) - \frac{\mu}{x^{2-\alpha}} \Phi(x) = \lambda \Phi(x),  \quad x \in (0,1), \\
&\Phi(0)=\Phi(1)=0,
\end{cases}
\end{equation}
that will be essential for our purposes.
For this reason, we first recall some results concerning the Bessel functions that will be useful in the rest of the paper (see \cite{Watson}, for more details).

For a real number $\nu \in \mathbb{R}_{+}$, we denote by $J_{\nu}$ the Bessel function of the first kind of order $\nu$
defined by the following Taylor series expansion around $x=0$:
$$
J_{\nu} (x) = \sum_{m\geq 0}  \frac{(-1)^{m}}{m! \;\Gamma(1+\nu+m)} \Big(\frac{x}{2}\Big)^{2m+\nu},
$$
where $\Gamma(.)$ is the Gamma function.

We recall that the Bessel function $J_{\nu}$ satisfies the following differential
equation
$$
x^2 y''(x) + x y'(x) + (x^2 - \nu^2) y(x) = 0, \qquad  x\in (0, + \infty).
$$
Moreover, the function $J_{\nu}$ has an infinite number of real zeros which are
simple with the possible exception of $x=0$ (see \cite{Lebedev,Elbert2001}). We denote by $(j_{\nu,n})_{n\geq 1}$ the strictly
increasing sequence of the positive zeros of $J_{\nu}$:
$$ 0<j_{\nu,1} < j_{\nu,2} < \cdots<j_{\nu,n}< \cdots$$
and we recall that
$$j_{\nu,n} \rightarrow +\infty\quad \text{as}\quad n\rightarrow +\infty$$
and the following bounds on the zeros $j_{\nu, n}$, which are provided in \cite{LM2008}:
\begin{itemize}
\item $\forall \nu \in \Big[0, \dfrac{1}{2}\Big],\, \forall n \geq 1$,
\begin{equation}\label{boundcase1}
\big( n + \frac{\nu}{2} - \frac{1}{4} \big) \pi \leq j_{\nu, n } \leq \big( n + \frac{\nu}{4} - \frac{1}{8} \big) \pi.
\end{equation}
\item $\forall \nu \geq \dfrac{1}{2},\, \forall n \geq 1$,
\begin{equation}\label{boundcase2}
\big( n + \frac{\nu}{4} - \frac{1}{8} \big) \pi \leq j_{\nu, n } \leq \big( n + \frac{\nu}{2} - \frac{1}{4} \big) \pi.
\end{equation}
\end{itemize}
In our analysis, we will need the following classical result (see \cite[Proposition 7.8]{KL2005}):
\begin{lemma}\label{lemmadifference}
Let $j_{\nu,n}, n\geq 1$ be the positive zeros of the Bessel function $J_{\nu}$. Then, the
following hold:
\begin{itemize}
\item If $\nu \in \Big[0, \dfrac{1}{2}\Big]$, the difference sequence ($ j_{\nu, n+1} - j_{\nu, n})_{n}$ is nondecreasing and
converges to $\pi$ as $n \longrightarrow +\infty.$
\item If $\nu \geq \dfrac{1}{2}$, the sequence $(j_{\nu, n+1} - j_{\nu, n})_{n}$ is nonincreasing and
converges to $\pi$ as $n \longrightarrow +\infty$.
\end{itemize}
\end{lemma}
We also have that the Bessel functions enjoy the following integral formula (see \cite{Watson}):
$$
\int_0^1 x  J_{\nu} (j_{\nu, n} x)  J_{\nu}(j_{\nu, m} x)\,dx = \frac{\delta_{nm}}{2} [J_{\nu}^{'}(j_{\nu, n})]^2, \quad n,m \in \mathbb{N}^{*},
$$
where, $\delta_{nm}$ is the Kronecker symbol.

Next, given $\mu\leq \mu(\alpha)$, let us introduce the quantity
\begin{equation*}
\nu(\alpha,\mu):= \dfrac{2}{2-\alpha} \sqrt{\Big( \frac{1-\alpha}{2} \Big)^2-\mu}=\dfrac{2}{2-\alpha} \sqrt{\mu(\alpha)-\mu}.
\end{equation*}
With the previous notation, we have the following result on the expression of the
eigenvalues and eigenfunctions related to problem \eqref{eigenvaluesproblem}
that have been computed in \cite{biccsantavan2020}:
\begin{proposition}\label{spectrum} Assume $\mu\leq  \mu(\alpha)$.
Then the admissible eigenvalues $\lambda$ for problem \eqref{eigenvaluesproblem} are given by
\begin{equation}\label{eigenvalues}
\lambda_{\alpha, \mu, n} = \Big( \dfrac{2-\alpha}{2}  \Big)^2 (j_{\nu(\alpha,\mu),n})^2, \qquad  \forall n \geq 1.
\end{equation}
and the associated normalized (in $L^2(0,1)$) eigenfunctions are
\begin{equation}\label{eigenfunctions}
\Phi_{\alpha,\mu, n}(x) =  \frac{\sqrt{2-\alpha}}{|J_{\nu(\alpha,\mu)}'(j_{\nu(\alpha,\mu), n})|}
x^{\frac{1-\alpha}{2}} J_{\nu(\alpha,\mu)}\Big(j_{\nu(\alpha,\mu), n}x^{\frac{2-\alpha}{2}}\Big), \quad n \geq 1 .
\end{equation}
Moreover, the family $(\Phi_{\alpha,\mu, n})_{n\geq 1} $ forms an orthonormal basis of $L^2(0,1)$.
\end{proposition}
We have the following result which will be used later.

\begin{lemma}\label{gapresult}
Let $(\lambda_{\alpha,\mu, k})_{k\geq 1}$ be the sequence of eigenvalues of the spectral problem \eqref{eigenvaluesproblem}.
Then, the following properties hold:
\begin{enumerate}
\item For all $n,m \in \mathbb{N}^\star$, there is a constant $\rho > 0$ such that 
$(\lambda_{\alpha,\mu, k})_{k\geq 1}$  satisfies the following gap condition:  there is a constant $\rho> 0$ such that
\begin{equation}\label{gap}
| \lambda_{\alpha,\mu, n} - \lambda_{\alpha,\mu, m} | \geq \rho |n^2 - m^2|, \qquad \forall n,m \geq 1.
\end{equation}
\item The series $\displaystyle\sum_{n\geq 1}  \frac{1}{\lambda_{\alpha,\mu,n}}$ is convergent.
\end{enumerate}
\end{lemma}
\begin{proof}
\begin{enumerate}
\item Let $n, m \in \mathbb{N}^\star$ with $n \geq m$. We have
\begin{align*}
&\lambda_{\alpha,\mu, n} - \lambda_{\alpha,\mu, m} = \Big(\dfrac{2-\alpha}{2}\Big)^2( j_{\nu(\alpha,\mu), n}^2 -  j_{\nu(\alpha,\mu), m}^2 )  \notag\\
&=\Big(\dfrac{2-\alpha}{2}\Big)^2 ( j_{\nu(\alpha,\mu), n} -  j_{\nu(\alpha,\mu), m} ) ( j_{\nu(\alpha,\mu), n} +  j_{\nu(\alpha,\mu), m} ) \notag\\
&=\Big(\dfrac{2-\alpha}{2}\Big)^2 \Big(( j_{\nu(\alpha,\mu), n} -  j_{\nu(\alpha,\mu), n-1})+ \cdots + (j_{\nu(\alpha,\mu), m+1} -  j_{\nu(\alpha,\mu), m} ) \Big)
(j_{\nu(\alpha,\mu), n} +  j_{\nu(\alpha,\mu), m}).
\end{align*}	
We can now distinguish the two different cases $\nu(\alpha,\mu) \in \Big[0, \dfrac{1}{2}\Big]$ and $\nu(\alpha,\mu) \geq \dfrac{1}{2}$, depending on the parameter $\mu$.
\begin{itemize}
\item if $\nu(\alpha,\mu) \in \Big[0, \dfrac{1}{2}\Big]$ $\Big($i.e. $\mu \in \Big( \dfrac{\alpha}{16}(3\alpha-4),\mu(\alpha) \Big] \Big)$, by virtue of Lemma \ref{lemmadifference}
  we immediately have that
$$j_{\nu(\alpha,\mu), n} -  j_{\nu(\alpha,\mu), n-1}  \geq  j_{\nu(\alpha,\mu), 2} -  j_{\nu(\alpha,\mu), 1}, \quad \forall n\geq 2.$$
Therefore,
\begin{equation*}
\lambda_{\alpha,\mu, n} - \lambda_{\alpha,\mu, m} \geq
(n-m) (j_{\nu(\alpha,\mu), 2} -  j_{\nu(\alpha,\mu), 1}) (j_{\nu(\alpha,\mu), n} +  j_{\nu(\alpha,\mu), m}).
\end{equation*}	
By using \eqref{boundcase1}, the last inequality becomes:
\begin{equation}\label{inequality}
\lambda_{\alpha,\mu, n} - \lambda_{\alpha,\mu, m} \geq \frac{7}{8} \pi^2 \Big(\dfrac{2-\alpha}{2}\Big)^2 (n-m) \big(n+m + \nu(\alpha,\mu) -\frac{1}{2} \big).
\end{equation}
Moreover, we have
$$
\big(n+m + \nu(\alpha,\mu) -\frac{1}{2} \big)>\frac{n+m}{2},
$$
and thus, that there exists $\rho= \frac{7}{64} \pi^2 (2-\alpha)^2 $ such that
\begin{align*}
\lambda_{\alpha,\mu, n} - \lambda_{\alpha,\mu, m} \geq \rho (n^2-m^2).
\end{align*}	
\item Let us now see the case $\nu(\alpha,\mu) \geq \dfrac{1}{2}$ \Big(i.e. $\mu \leq \dfrac{\alpha}{16}(3\alpha-4) $\Big). Here we use the fact that the sequence $(j_{\nu(\alpha,\mu), n+1} - j_{\nu(\alpha,\mu), n})_{n}$ is nonincreasing and converges to $\pi$. This ensures that
\begin{equation*}
  j_{\nu(\alpha,\mu), n+1} - j_{\nu(\alpha,\mu), n} \geq \pi,\qquad \forall n\geq 1.
\end{equation*}
Therefore:
\begin{equation*}
\lambda_{\alpha,\mu, n} - \lambda_{\alpha,\mu, m} \geq  \Big(\dfrac{2-\alpha}{2}\Big)^2 \pi (n-m) (j_{\nu(\alpha,\mu), n} +  j_{\nu(\alpha,\mu), m}).
\end{equation*}
Owing to \eqref{boundcase2}, we also have
\begin{equation*}
j_{\nu(\alpha,\mu), n} +  j_{\nu(\alpha,\mu), m} \geq \big(n+m + \frac{\nu(\alpha,\mu)}{2} -\frac{1}{4} \big) \pi \geq \pi (n+m).
\end{equation*}
Combining the above last two estimates, the thesis follows with $\rho= \Big(\dfrac{2-\alpha}{2}\Big)^2\pi^2$.
\end{itemize}
Thus, in every case there holds
\begin{equation*}
\lambda_{\alpha,\mu, n} - \lambda_{\alpha,\mu, m} \geq \rho (n^2-m^2).
\end{equation*}
In both cases, after reversing the roles of $n$ and $m$, one has
\begin{equation*}
\lambda_{\alpha,\mu, m} - \lambda_{\alpha,\mu, n} \geq \rho (m^2-n^2).
\end{equation*}
Hence,
\begin{equation*}
| \lambda_{\alpha,\mu, n} - \lambda_{\alpha,\mu, m} | \geq \rho |n^2 - m^2 |, \qquad \forall n,m \geq 1,
\end{equation*}
for a constant $\rho>0$.

\item This point follows easily form \eqref{boundcase1}. Indeed, $\forall \nu(\alpha,\mu) \in \Big(0,\dfrac{1}{2}\Big],\quad \forall n \geq 1$
\begin{equation*}
\Big( n-\dfrac{1}{4} \Big) \pi \leq j_{\nu(\alpha,\mu),n}
\end{equation*}
Thus 
\begin{equation*}
\sum\limits_{n \geq 1} \dfrac{1}{\lambda_{\alpha,\mu,n}}
\leq  \dfrac{1}{\pi^2} \Big(\frac{2}{2-\alpha}\Big)^2 \sum\limits_{n \geq 1} \dfrac{1}{ \Big(n - \frac{1}{4} \Big)^2 }
\leq  \dfrac{4}{\pi^2} \Big(\frac{2}{2-\alpha}\Big)^2 \sum\limits_{n \geq 1} \dfrac{1}{ n^2 } < +\infty.
\end{equation*}
\end{enumerate}
\end{proof}
\subsection{Vectorial degenerate/singular operators}
Let $A$ be given by \eqref{AandB} and consider the degenerate/singular vectorial operator
\begin{equation}\label{operatorL}
\begin{aligned}
L: D(L) \subset L^2(0,1)^2 &\rightarrow L^2(0,1)^2\\
y&\mapsto  - (x^\alpha y_x)_x - \frac{\mu}{x^{2-\alpha}}y - Ay,
\end{aligned}
\end{equation}
with domain $D(L)=H^{2,\mu}_\alpha(0,1)^2 \cap H^{1,\mu}_{\alpha,0}(0,1)^2$ and also its adjoint $L^*$.

In the sequel, we pass to derive some properties of the eigenvalues and eigenfunctions of the operators $L$ and $L^*$ which
will be useful for developing the moment method. Let us first analyze the spectrum of the operators $L$ and $L^*$:

\begin{proposition}\label{propo-eigenf-LL*}
Let us consider the operator $L$ given by \eqref{operatorL} and its
adjoint $L^*$. Then,
\begin{enumerate}
\item
The spectra of $L$ and $L^*$ are given by
$\sigma(L) = \sigma(L^*) = \big\{ \lambda_{\alpha,\mu, n}^{(1)} , \lambda_{\alpha,\mu, n}^{(2)} \big\}_{n\geq1}$
with
\begin{equation}\label{sigmaLL*}
\lambda_{\alpha,\mu, n}^{(1)} = \lambda_{\alpha,\mu, n} - \alpha_1,\quad \lambda_{\alpha,\mu, n}^{(2)}= \lambda_{\alpha,\mu, n} - \alpha_2,\quad \forall n\geq1,
\end{equation}
where $\alpha_1$ and $\alpha_2$ are the eigenvalues of the matrix $A$ defined by :
\begin{itemize}
\item Case 1: $a_2^2 + 4a_1>0 $,
 \begin{equation}\label{mu1mu2real}
\alpha_1 = \frac{1}{2} \Big(a_2 - \sqrt{a_2^2 + 4a_1}\Big)  \quad \text{and} \quad \alpha_2 = \frac{1}{2} \Big(a_2 + \sqrt{a_2^2 + 4a_1}\Big).
\end{equation}
\item Case 2: $a_2^2 + 4a_1<0 $,
\begin{equation}\label{mu1mu2complex}
\alpha_1 = \frac{1}{2} \Big(a_2 + i \sqrt{-(a_2^2 + 4a_1)}\Big)  \; \text{and} \; \alpha_2 = \frac{1}{2}  \Big(a_2 - i \sqrt{-(a_2^2 + 4a_1)} \Big).
\end{equation}
\end{itemize}
\item
For each $n \geq 1$, the corresponding eigenfunctions of $L$ (resp., $L^*$) associated to $\lambda_{\alpha,\mu, n}^{(1)}$ and $\lambda_{\alpha,\mu, n}^{(2)}$ are respectively given by
\begin{equation}\label{eigenf-L}
\psi_n^{(1)}  =U_{1} \Phi_{\alpha,\mu, n}, \qquad \psi_n^{(2)}=U_{2} \Phi_{\alpha,\mu, n},
\end{equation}
with
\begin{equation*}
 U_{1}=  \frac{a_1}{\alpha_2^2 + a_1} \begin{pmatrix} -\alpha_2 \\ 1 \end{pmatrix}\quad \text{and}\quad U_{2}= \begin{pmatrix} -\alpha_1 \\ 1 \end{pmatrix}
\end{equation*}
(resp.,
\begin{equation}\label{eigenf-L*}
\Psi_n^{(1)}  = V_{1}\Phi_{\alpha,\mu, n}, \qquad \Psi_n^{(2)}  = V_{2}\Phi_{\alpha,\mu, n},
\end{equation}
with
\begin{equation*}
 V_{1}=  \begin{pmatrix} -\frac{\alpha_2}{a_1} \\ 1 \end{pmatrix}\quad \text{and}\quad V_{2}= \frac{a_1}{\alpha_1^2 + a_1} \begin{pmatrix} -\frac{\alpha_1}{a_1} \\ 1 \end{pmatrix}.
\end{equation*}
\end{enumerate}
\end{proposition}
\begin{proof}
We will prove the result for the operator $L$. The same reasoning provides the proof for its
adjoint $L^*$.

Using the fact that the function $\Phi_{\alpha,\mu, n}$ is the eigenfunction of the Dirichlet degenerate/singular operator $ -\partial_x(x^\alpha \partial_{x}. ) -\frac{\mu}{x^{2-\alpha}}$ associated to the eigenvalue $\lambda_{\alpha,\mu, n}$, one can see that the eigenvalues of the operator $L$ correspond to the eigenvalues of the matrices
\begin{equation*}
\lambda_{\alpha,\mu, n} Id - A, \qquad \forall n \geq 1,
\end{equation*}
$(Id\in \mathcal{L}(\mathbb{R}^2)$ is the identity matrix) and the associated eigenfunctions of $L$ are given under the form $\psi_n(\cdot) = z_n \Phi_{\alpha,\mu, n}(\cdot)$, where $z_n \in \mathbb{R}^2$ is the
associated eigenvector of the matrix $\lambda_{\alpha,\mu, n} Id - A$.

Taking into account the expression of the characteristic polynomial of $\lambda_{\alpha,\mu, n} Id - A$:
$$P(z) = z^2 - z (2 \lambda_{\alpha,\mu, n} - a_2) + \lambda_{\alpha,\mu, n} (\lambda_{\alpha,\mu, n} - a_2) - a_1, \qquad n \geq 1,$$
a direct computation provides the formulas \eqref{sigmaLL*} and \eqref{eigenf-L} as eigenvalues and associated eigenfunctions
of the operator $L$. This ends the proof.
\end{proof}

Let us now check that the sequence of eigenvalues of $L$
and $L^*$ fulfills the conditions in Theorem \ref{thm-biorth}. One has
\begin{proposition}\label{prop-incres-seq}
Assume that the following condition holds
\begin{equation}\label{spectre-cond}
\lambda_{\alpha,\mu, n}- \lambda_{\alpha,\mu, l} \neq \alpha_1 - \alpha_2,\quad \forall n, l \in \mathbb{N}^*, \quad \text{with} \quad n\neq l.
\end{equation}
Then, one can construct a family from the spectrum $\big\{ \lambda_{\alpha,\mu, n}^{(1)} , \lambda_{\alpha,\mu, n}^{(2)} \big\}_{n\geq1}$ defined by
\begin{equation}\label{increas-sequence}
\begin{aligned}
\big\{ \Lambda_{\alpha,\mu, n} \big\}_{n\geq1} &= \big\{ \lambda_{\alpha,\mu, n}^{(1)} + \alpha_2 , \lambda_{\alpha,\mu, n}^{(2)}+\alpha_2 \big\}_{n\geq1} \\
&=\{\lambda_{\alpha,\mu, n} + \alpha_2 - \alpha_1: n\geq 1\} \cup \{\lambda_{\alpha,\mu, n}: n \geq 1\},
\end{aligned}
\end{equation}
which satisfies the hypotheses in Theorem \ref{thm-biorth}.
\end{proposition}

\begin{proof}
We distinguish between three cases depending on the spectrum of the matrix $A$.\\

\noindent\textbf{Case 1: $A$ has two real eigenvalues $\alpha_1$ and $\alpha_2$, chosen such that $\alpha_1<\alpha_2$.}

Let us introduce the sequence $\big\{\Lambda_{\alpha,\mu, n} \big\}_{n\geq1}$, where
\begin{equation*}
\big\{ \Lambda_{\alpha,\mu, n}: n\geq1 \big\}
:= \big\{ \lambda_{\alpha,\mu, n}^{(1)} + \alpha_2 , \lambda_{\alpha,\mu, n}^{(2)}+\alpha_2 \big\}_{n\geq1}.
\end{equation*}
The hypothesis 1) holds true if and only if the condition \eqref{spectre-cond} is satisfied. In addition, the hypotheses 2) and 3) are obviously satisfied by definition.

Let us now show the hypothesis 4). Since $\alpha_2 - \alpha_1 >0$, observe that $\big\{  \lambda_{\alpha,\mu, n}^{(1)} + \alpha_2 \big\}_{n\geq1}$ and
$\big\{\lambda_{\alpha,\mu, n}^{(2)}+\alpha_2 \big\}_{n\geq1}$ are increasing sequences satisfying
\begin{equation*}
0<\lambda_{\alpha,\mu, n}^{(2)}+\alpha_2< \lambda_{\alpha,\mu, n}^{(1)}+\alpha_2,\quad \forall n\geq1.
\end{equation*}
Thus, we deduce that the sequence $\big\{\Lambda_{\alpha,\mu, n}\big\}_{n\geq1}$ can be rearranged into a positive increasing sequence.

Let us move to prove hypothesis 5). For this purpose, we are going to give an explicit rearrangement of the sequence
$\big\{ \lambda_{\alpha,\mu, n}^{(1)} + \alpha_2 , \lambda_{\alpha,\mu, n}^{(2)}+\alpha_2 \big\}_{n\geq1}$. Firstly, observe that there exists an integer $n_0\geq1$ and a constant $C>0$ such that
\begin{equation}\label{mainineq}
\begin{aligned}
&\lambda_{\alpha,\mu, n-1}^{ (1)} <  \lambda_{\alpha,\mu, n}^{(2)} < \lambda_{\alpha,\mu, n}^{(1)} <  \lambda_{\alpha,\mu, n+1}^{(2)}<\cdots,\quad \forall n\geq n_0,\;\text{and}\\
& \min_{n\geq n_0}\big\{\lambda_{\alpha,\mu, n}^{(2)}- \lambda_{\alpha,\mu, n-1}^{ (1)}, \lambda_{\alpha,\mu, n}^{(1)} - \lambda_{\alpha,\mu, n}^{(2)}\big\}>C.
\end{aligned}
\end{equation}
Indeed, using \eqref{gap}, one has
\begin{align}\label{0mainineq}
\lambda_{\alpha,\mu, n}^{(2)} -  \lambda_{\alpha,\mu, n-1}^{ (1)}
&=  \lambda_{\alpha,\mu, n} -  \lambda_{\alpha,\mu, n-1} + \alpha_1 -\alpha_2 \notag\\
&\geq  \rho (2n-1) + \alpha_1 -\alpha_2  \underset{n \to +\infty} {\longrightarrow} +\infty.
\end{align}
From \eqref{0mainineq} and the fact that $\lambda_{\alpha,\mu, n}^{(1)} - \lambda_{\alpha,\mu, n}^{(2)} = \alpha_2 - \alpha_1>0$, we can conclude \eqref{mainineq}.

Therefore, if $1\leq n \leq 2n_0 -2$, we define $\Lambda_{\alpha,\mu,n}$ such that
\begin{equation*}
\begin{aligned}
\big\{ \Lambda_{\alpha,\mu, n} \big\}_{1\leq n \leq 2n_0 -2} &= \{\lambda_{\alpha,\mu, n}^{(1)} + \alpha_2\}_{1\leq n\leq n_0-1} \cup \{\lambda_{\alpha,\mu, n}^{(2)}+ \alpha_2\}_{1\leq n\leq n_0-1}\quad \text{and}\\
&\quad \Lambda_{\alpha,\mu, n} < \Lambda_{\alpha,\mu, n+1}\quad \forall n: 1\leq n \leq 2n_0 -3.
\end{aligned}
\end{equation*}
Moreover, from $(2n_0-1)$-th term, we choose to arrange the sequence as follows:
\begin{equation}\label{maindef}
\Lambda_{\alpha,\mu, 2n-1} = \lambda_{\alpha,\mu, n}^{(2)} + \alpha_2 \quad \text{and}\quad \Lambda_{\alpha,\mu, 2n}= \lambda_{\alpha,\mu, n}^{(1)} + \alpha_2,\; \forall n \geq n_0.
\end{equation}
Since the elements of the sequence $\big\{ \Lambda_{\alpha,\mu, n} \big\}_{n\geq 1}$ are pairwise different and from  \eqref{mainineq}, one has:
\begin{equation}\label{mainineq2}
\inf\limits_{n, m \geq 1: n \neq m} |\lambda_{\alpha,\mu, n}^{(1)}  - \lambda_{\alpha,\mu, m}^{(2)} | > 0.
\end{equation}
Hence, thanks to \eqref{mainineq2}, the sequence $\big\{ \Lambda_{\alpha,\mu, n} \big\}_{n\geq 1}$ satisfies the second inequality in \eqref{strong-gap} for every $q\geq1$.

Our next task will be to prove the first inequality of \eqref{strong-gap} for appropriate $q>0$ and $\varrho>0$.
To this aim, as it has been remarked in \cite{BandN}, it is enough to prove the existence of $q>0$ and $\tilde{\varrho}>0$ such that
\begin{equation}\label{mainineq3}
|\Lambda_n - \Lambda_m| \geq \tilde{\varrho} |n^2 - m^2| \quad \forall n, m\geq q, \; |n -m|\geq q.
\end{equation}

We divide the proof of \eqref{mainineq3} into two steps.
\begin{enumerate}
\item  Observe that, if $n, m \in \mathbb{N}^\star$ are such that $n, m \geq n_0$ and $|n -m|\geq n_0$, then by \eqref{maindef} and using \eqref{gap} we have
\begin{equation*}
|\Lambda_{\alpha,\mu,2n} - \Lambda_{\alpha,\mu,2m}| =|\lambda_{\alpha,\mu,n} - \lambda_{\alpha,\mu,m}| \geq \rho |n^2-m^2| = \frac{\rho}{4} |(2n)^2- (2m)^2|
\end{equation*}
and
\begin{equation*}
\begin{aligned}
|\Lambda_{\alpha,\mu,2n-1} -\Lambda_{\alpha,\mu,2m-1}| =|\lambda_{\alpha,\mu,n} - \lambda_{\alpha,\mu,m}| &\geq \rho |n^2-m^2| \\
&\geq \frac{\rho}{4} |(2n-1)^2 -(2m-1)^2|.
\end{aligned}
\end{equation*}
We obtain thus the proof of \eqref{mainineq3} for $q=n_0$ and $\tilde{\varrho}= \frac{\rho}{4}$.
\item Let $n, m \in \mathbb{N}^\star$ such that $n, m \geq n_0$. From \eqref{maindef}, by denoting $\tilde{n} = 2n$ and $\tilde{m} = 2m-1 $ and using again \eqref{gap}, we readily see that
\begin{align*}
|\Lambda_{\alpha,\mu,\tilde{n}} - \Lambda_{\alpha,\mu,\tilde{m}}| &=\big|\lambda_{\alpha,\mu, n}^{(1)} - \lambda_{\alpha,\mu, m}^{(2)} \big|\notag\\
&=\big|\lambda_{\alpha,\mu,n} - \lambda_{\alpha,\mu,m} +(\alpha_2- \alpha_1)\big|\notag\\
& \geq \rho |n^2-m^2|- (\alpha_2- \alpha_1)\notag\\
& = \frac{\rho}{4} |\tilde{n}^2 - (\tilde{m}+1)^2|- (\alpha_2- \alpha_1)\notag\\
& = \frac{\rho}{4} |\tilde{n}^2 -  \tilde{m}^2 - 2\tilde{m}-1|- (\alpha_2- \alpha_1).
\end{align*}
Now, observe that if  $\tilde{n} < \tilde{m}$, we have
\begin{align*}
|\Lambda_{\alpha,\mu,\tilde{n}} - \Lambda_{\alpha,\mu,\tilde{m}}|
&\geq\frac{\rho}{4} (\tilde{m}^2 -  \tilde{n}^2 ) \Big(1 - \frac{4(\alpha_2- \alpha_1)}{\rho(\tilde{m}^2 -  \tilde{n}^2 )} \Big).
\end{align*}
Let us take an integer $q_0\geq \max\{2n_0 -1,\frac{4(\alpha_2- \alpha_1)}{\rho}\}$. Then,
$\forall \tilde{m}, \tilde{n} \geq q_0$ with $ |\tilde{m}- \tilde{n}| \geq q_0$, one has
\begin{align*}
|\Lambda_{\alpha,\mu,\tilde{n}} - \Lambda_{\alpha,\mu,\tilde{m}}|
&\geq \frac{\rho}{4} (\tilde{m}^2 -  \tilde{n}^2 ) \Big(1 - \frac{4(\alpha_2- \alpha_1)}{\rho(\tilde{m} + \tilde{n})q_0} \Big)\notag\\
&\geq\frac{\rho}{4} (\tilde{m}^2 -  \tilde{n}^2 ) \Big(1 - \frac{2(\alpha_2- \alpha_1)}{\rho q_0} \Big)\notag\\
&\geq\frac{\rho}{8} (\tilde{m}^2 -\tilde{n}^2).
\end{align*}
On the other hand, if  $\tilde{n} > \tilde{m}$, we have
\begin{align*}
|\Lambda_{\alpha,\mu,\tilde{n}} - \Lambda_{\alpha,\mu,\tilde{m}}|
&\geq\frac{\rho}{4} (\tilde{n}^2 -  \tilde{m}^2 ) \Big(1 - \Big( \frac{4(\alpha_2- \alpha_1)}{\rho} +2 \tilde{m} + 1 \Big)\frac{1}{(\tilde{n}^2 -  \tilde{m}^2 )} \Big).
\end{align*}
Let us work with an integer $q_1$ given by
$$q_1\geq \max\{2n_0 -1,\frac{4(\alpha_2- \alpha_1)}{\rho}+4\}.$$
Thus, if $\tilde{n}, \tilde{m} \in \mathbb{N}^*$
are such that $\tilde{n}, \tilde{m} \geq q_1$ and  $|\tilde{n}- \tilde{m}| \geq q_1$, then one has
\begin{align*}
|\Lambda_{\alpha,\mu,\tilde{n}} - \Lambda_{\alpha,\mu,\tilde{m}}|
&\geq\frac{\rho}{4} (\tilde{n}^2 -  \tilde{m}^2 ) \Big(1 - \Big(\frac{4(\alpha_2- \alpha_1)}{\rho} +2 \tilde{m} + 1\Big)\frac{1}{2\tilde{m}q_1} \Big)\notag\\
&\geq\frac{\rho}{4} (\tilde{n}^2 -  \tilde{m}^2 ) \Big(1 - \frac{1}{q_1} \Big(\frac{2(\alpha_2- \alpha_1)}{\rho} + 2 \Big) \Big)\notag\\
&\geq\frac{\rho}{8} (\tilde{n}^2 -  \tilde{m}^2 ).
\end{align*}
Hence, choosing $q=max\{q_0,q_1\}$, \eqref{mainineq3} follows immediately for $\tilde{\varrho}= \frac{\rho}{8}$.
\end{enumerate}
In conclusion, we have proved the existence of a number $q\geq 1$ such that \eqref{mainineq3} holds.

Let us now show the hypothesis 6). From the definition of  $\big\{\Lambda_{\alpha,\mu, n} \big\}_{n\geq 1}$, for any $r > 0$, we can write:
\begin{align*}
\mathcal{N}(r) & = \#\{ k : \; \lambda_{\alpha,\mu, k} + \alpha_2- \alpha_1 \leq r \} + \#\{ k : \; \lambda_{\mu, k} \leq r \}  \\
& =  \# \mathcal{A}_1(r) +  \# \mathcal{A}_2(r) = n_1 + n_2,
\end{align*}
where $\mathcal{A}_i(r)=\{k : \; \lambda_{\alpha,\mu, k}^{(i)}+\alpha_2 \leq r \}$ and $n_i= \# \mathcal{A}_i(r)$, i=1,2.
Our purpose is to prove suitable estimates for $ n_1$ and $n_2$.

From the definition of $\mathcal{A}_2(r)$ and $n_2$, we deduce that $n_2$ is a natural number which is characterized by $\lambda_{\alpha,\mu, n_2} \leq r$ and $\lambda_{\alpha,\mu, n_2+1} > r$.
We distinguish two cases depending on the value of $\nu_{\mu}$. Let us start by the case $\nu(\alpha,\mu) \leq \frac{1}{2}$.
From the inequality $\lambda_{\alpha,\mu, n_2} \leq r$ and by \eqref{boundcase1}, we have
$\big( n_2 + \frac{\nu(\alpha,\mu)}{2} - \frac{1}{4} \big)^2 \pi^2 \leq r$
so that
\begin{equation}\label{wupern1}
n_2 \leq \frac{\sqrt{r}}{ \pi} - \frac{\nu(\alpha,\mu)}{2} + \frac{1}{4}.	
\end{equation}
On the other hand, from the inequality $ \lambda_{\alpha,\mu, n_2+1} > r$, we get
\begin{equation*}
n_2 > \frac{\sqrt{r}}{ \pi} - \frac{\nu(\alpha,\mu)}{4} - \frac{7}{8}.	
\end{equation*}
Summarizing, $n_2$ is a nonnegative integer such that
\begin{equation}\label{wboundn1}
 \frac{\sqrt{r}}{ \pi} - \frac{\nu(\alpha,\mu)}{4} - \frac{7}{8} < n_2 \leq \frac{\sqrt{r}}{ \pi} - \frac{\nu(\alpha,\mu)}{2} + \frac{1}{4},\quad \forall r>0.	
\end{equation}
Next we are going to estimate $n_1$. Using arguments similar to the ones used above,
we can see that
$$ \lambda_{\alpha,\mu, n_1} +\alpha_2- \alpha_1 \leq r $$
and
$$ \lambda_{\alpha,\mu, n_1 +1} +\alpha_2- \alpha_1 > r $$
imply that
\begin{equation*}
\frac{\sqrt{r+\alpha_1-\alpha_2}}{ \pi} - \frac{\nu(\alpha,\mu)}{4} - \frac{7}{8} < n_1 \leq \frac{\sqrt{r+\alpha_1-\alpha_2}}{ \pi} - \frac{\nu(\alpha,\mu)}{2} + \frac{1}{4}.	
\end{equation*}
Then, using the fact that $\sqrt{a} - \sqrt{b} \leq \sqrt{a -b} $ and $\sqrt{a-b} \leq \sqrt{a}$  provided $a \geq b>0$, one gets
\begin{equation}\label{wboundn2}
\frac{\sqrt{r}}{ \pi} - \frac{\sqrt{\alpha_2- \alpha_1}}{ \pi} - \frac{\nu(\alpha,\mu)}{4} - \frac{7}{8} < n_1 \leq \frac{\sqrt{r}}{\pi} - \frac{\nu(\alpha,\mu)}{2} + \frac{1}{4},\quad \forall r>0.
\end{equation}
Recall that $\mathcal{N}(r) = n_1 + n_2$. Thus, combining \eqref{wboundn1} and \eqref{wboundn2}, it follows that for  $\nu(\alpha,\mu) \leq \frac{1}{2}$:
\begin{equation*}
\frac{2\sqrt{r}}{ \pi} - \frac{\sqrt{\alpha_2- \alpha_1}}{ \pi} - \frac{\nu(\alpha,\mu)}{2} - \frac{7}{4} < \mathcal{N}(r) \leq \frac{2\sqrt{r}}{ \pi} - \nu(\alpha,\mu) + \frac{1}{2},\quad \forall r>0,
\end{equation*}
and deduce \eqref{ineq-counting} with
\begin{equation*}
p =  \frac{2}{\pi}\quad \text{and}\quad s= \max\{  \frac{\sqrt{\alpha_2- \alpha_1}}{ \pi} +\frac{\nu(\alpha,\mu)}{2} +\frac{7}{4}, \; -\nu(\alpha,\mu) + \frac{1}{2}	\} =  \frac{\sqrt{\alpha_2- \alpha_1}}{ \pi} +\frac{\nu(\alpha,\mu)}{2} +\frac{7}{4}.
\end{equation*}
The case  $\nu_{\mu} \geq \frac{1}{2}$ can be treated in a similar way, but, instead of working with the bounds \eqref{boundcase1}, we will use \eqref{boundcase2} to obtain
\begin{equation}\label{sboundn1}
\frac{\sqrt{r}}{ \pi} - \frac{\nu(\alpha,\mu)}{2} - \frac{3}{4} < n_2 \leq \frac{\sqrt{r}}{ \pi} - \frac{\nu(\alpha,\mu)}{4} + \frac{1}{8},\quad \forall r>0,		
\end{equation}
and
\begin{equation}\label{sboundn2}
\frac{\sqrt{r}}{ \pi} - \frac{\sqrt{\alpha_2- \alpha_1}}{ \pi} - \frac{\nu(\alpha,\mu)}{2} - \frac{3}{4} < n_1 \leq \frac{\sqrt{r}}{ \pi} - \frac{\nu(\alpha,\mu)}{4} + \frac{1}{8},\quad \forall r>0.
\end{equation}
From the inequalities \eqref{sboundn1} and \eqref{sboundn2}, we obtain that:
\begin{equation*}
\frac{2\sqrt{r}}{ \pi} - \frac{\sqrt{\alpha_2- \alpha_1}}{ \pi} - \nu(\alpha,\mu) - \frac{3}{2} < \mathcal{N}(r) \leq \frac{2\sqrt{r}}{ \pi} - \frac{\nu(\alpha,\mu)}{2}  + \frac{1}{4},\quad \forall r>0,
\end{equation*}
and again deduce \eqref{ineq-counting} with
\begin{equation*}
p =  \frac{2}{\pi}\quad \text{and}\quad s= \max\{  \frac{\sqrt{\alpha_2- \alpha_1}}{ \pi} +\nu(\alpha,\mu) +\frac{3}{2}, \; -\frac{\nu(\alpha,\mu)}{2} + \frac{1}{4}	\} = \frac{\sqrt{\alpha_2- \alpha_1}}{ \pi} +\nu(\alpha,\mu) +\frac{3}{2}.
\end{equation*}
We thus obtain the last hypothesis 6) of Theorem \ref{thm-biorth}.
This ends the proof in this case.

\noindent\textbf{Case 2: $A$ has two complex eigenvalues $\alpha_1$ and $ \alpha_2$.}

In this case $a_2^2 + 4a_1<0$,
$$
\alpha_1 = \frac{a_2 }{2} + i \beta, \quad \text{and}\quad \alpha_2 = \frac{a_2 }{2} - i \beta,
$$
where $\beta:= \frac{1}{2}  \sqrt{-(a_2^2 + 4a_1)}$.

Now, we consider the complex sequence $\{\Lambda_{\alpha,\mu,n}\}_{n\geq 1}$, with
\begin{equation}\label{maindef2}
\begin{aligned}
&\Lambda_{\alpha,\mu,2n-1}= \lambda_{\alpha,\mu, n}^{(2)}+\alpha_2 = \lambda_{\alpha,\mu,n},\quad \forall n \geq1,\\
&\Lambda_{\alpha,\mu,2n}= \lambda_{\alpha,\mu, n}^{(1)} + \alpha_2 =\lambda_{\alpha,\mu,n} - 2i\beta, \quad \forall n \geq1.
\end{aligned}
\end{equation}
Let us check if the hypotheses in Theorem \ref{thm-biorth} hold true for $\big\{ \Lambda_{\alpha,\mu, n} \big\}_{n\geq1}$.\\
First, it is clearly that the sequence $\big\{ \Lambda_{\alpha,\mu, n} \big\}_{n\geq1}$ always satisfies the hypothesis 1).
Furthermore, the hypothesis 2) follows directly from the fact that
$$
\Re (\Lambda_{\alpha,\mu,2n}) = \Re (\Lambda_{\alpha,\mu,2n-1}) = \lambda_{\alpha,\mu,n}>0. $$
The hypothesis 3) is clearly fulfilled. Indeed, one can find $\delta>0$ (which depends on $\beta$) such that
$$|\Im (\Lambda_{\alpha,\mu,2n})| =2 \beta \leq \delta  \sqrt{\Re (\Lambda_{\alpha,\mu,2n})}$$
and
$$|\Im (\Lambda_{\alpha,\mu,2n-1})|= 0 \leq \delta \sqrt{\Re (\Lambda_{\alpha,\mu,2n-1})}.$$
Let us now prove hypothesis 4).
To this end, it suffices to prove that there exists an integer $\tilde{n}_0\geq 1$ such that for all $n\geq \tilde{n}_0$ $|\Lambda_{\alpha,\mu,2n}| \leq |\Lambda_{\alpha,\mu,2n+1}|$. Using \eqref{gap}, we have
\begin{align*}
|\Lambda_{\alpha,\mu,2n+1}|^2 - |\Lambda_{\alpha,\mu,2n}|^2&= \lambda_{\alpha,\mu,n+1}^2 - \lambda_{\alpha,\mu,n}^2 - 4\beta^2\\
&\geq (\lambda_{\alpha,\mu,n+1} - \lambda_{\alpha,\mu,n})^2 - 4\beta^2\\
& \geq \rho^2 |(n+1)^2 - n^2|^2- 4\beta^2\\
& = \rho^2 (2n+1)^2 - 4\beta^2,
\end{align*}
which implies that
$$\displaystyle\lim_{n \rightarrow +\infty}\big(|\Lambda_{\alpha,\mu,2n+1}|^2 - |\Lambda_{\alpha,\mu,2n}|^2\big)=+\infty.$$
Therefore, there exists $\tilde{n}_0\geq 1$ such that $\{\Lambda_{\alpha,\mu,n}\}_{n\geq 2\tilde{n}_0} $ is nondecreasing in modulus. This shows hypothesis 4).

Let us now check if the hypothesis 5) holds true.
To this aim, we choose to arrange the sequence $\{\Lambda_{\alpha,\mu,n}\}_{n\geq1} $ defined in \eqref{maindef2} as
follows:
\begin{equation*}
\begin{aligned}
\big\{ \Lambda_{\alpha,\mu, n} \big\}_{1\leq n \leq 2\tilde{n}_0 -2} &= \{\lambda_{\alpha,\mu, n}^{(1)} + \alpha_2\}_{1\leq n\leq \tilde{n}_0-1} \cup \{\lambda_{\alpha,\mu, n}^{(2)}+ \alpha_2\}_{1\leq n\leq \tilde{n}_0-1}\quad \text{and}\\
&\quad |\Lambda_{\alpha,\mu, n}| < |\Lambda_{\alpha,\mu, n+1}|\quad \forall n: 1\leq n \leq 2\tilde{n}_0 -3.
\end{aligned}
\end{equation*}
Moreover, from $(2\tilde{n}_0-1)$-th term, we set:
\begin{equation}\label{maindef3}
\Lambda_{\alpha,\mu, 2n-1} = \lambda_{\alpha,\mu, n}^{(2)} + \alpha_2 \quad \text{and}\quad \Lambda_{\alpha,\mu, 2n}= \lambda_{\alpha,\mu, n}^{(1)} + \alpha_2,\; \forall n \geq \tilde{n}_0.
\end{equation}
First, observe that the second property is actually satisfied for any $q$.
Our next objective will be to prove the first inequality in \eqref{strong-gap}.
Arguing as done in the real case, by Lemma \ref{gapresult}, there exists $\rho >0$ such that
\begin{equation*}
|\Lambda_{\alpha,\mu,2n} - \Lambda_{\alpha,\mu,2m}|  \geq  \frac{\rho}{4} |(2n)^2-(2m)^2|,\quad \forall n,m\geq \tilde{n}_0
\end{equation*}
and
\begin{equation*}
|\Lambda_{\alpha,\mu,2n-1} -\Lambda_{\alpha,\mu,2m-1}| \geq \frac{\rho}{4} |(2n-1)^2 -(2m-1)^2|,\quad \forall n,m\geq \tilde{n}_0.
\end{equation*}
Moreover, denoting $\tilde{n} = 2n$ and $\tilde{m} = 2m-1 $, one can prove that there exists $ q \geq \max\{4, 2\tilde{n}_0-1\}$ such that $\forall \tilde{n}, \tilde{m} \geq q$ with $ |\tilde{n}- \tilde{m}| \geq q$, we have
\begin{equation*}
|\Lambda_{\alpha,\mu,\tilde{n}} - \Lambda_{\alpha,\mu,\tilde{m}}|^2=|\Lambda_{\alpha,\mu,2n} - \Lambda_{\alpha,\mu,2m-1}|^2
\geq \Big(\frac{\rho}{8} |\tilde{n}^2 -  \tilde{m}^2| \Big)^2.
\end{equation*}
Indeed, by \eqref{gap}, for $\tilde{n}, \tilde{m} \geq 2\tilde{n}_0-1$ we have
\begin{align*}
|\Lambda_{\alpha,\mu,\tilde{n}} - \Lambda_{\alpha,\mu,\tilde{m}}|^2&=|\Lambda_{\alpha,\mu,2n} - \Lambda_{\alpha,\mu,2m-1}|^2 \\
&= \big|\lambda_{\alpha,\mu,n} - \lambda_{\alpha,\mu,m}\big|^2 + 4\beta^2  \\
&\geq \big|\lambda_{\alpha,\mu,n} - \lambda_{\alpha,\mu,m}\big|^2\\
&\geq \Big(\rho |n^2 - m^2|\Big)^2
=\Big(\frac{\rho}{4} |\tilde{n}^2 -  \tilde{m}^2 - 2\tilde{m}-1|\Big)^2.
\end{align*}
Next, if  $|\tilde{n}- \tilde{m}| \geq 4$, simple computation gives
\begin{align*}
|\tilde{n}^2 -  \tilde{m}^2 - 2\tilde{m}-1| &= |(\tilde{n}^2 -  \tilde{m}^2)(1 - \frac{2\tilde{m}+1}{\tilde{n}^2 -  \tilde{m}^2})| \\
&\geq \frac{1}{2} |\tilde{n}^2- \tilde{m}^2|.
\end{align*}
Hence, the conclusion follows by working with $q$ given by
$$
q \geq \max\{4, 2\tilde{n}_0-1\}.
$$

Finally, proceeding as in the real case, it is not difficult to obtain some suitable parameters $p$ and $s$ for which the inequality \eqref{ineq-counting} holds.

\noindent\textbf{Case 3: $A$ has a double eigenvalue.}

In this case $a_2^2 + 4a_1=0$. We denote by $\alpha=\frac{a_2}{2}\in \mathbb{R}$ the eigenvalue of $A$. Thus, the sequence $\big\{ \Lambda_{\alpha,\mu, n} \big\}_{n\geq1}$ is then reduced to $\{\lambda_{\alpha,\mu, n}\}_{n\geq1}$. In view of Lemma \ref{gapresult}, and reasoning as in the first case, we  automatically conclude that $\big\{ \Lambda_{\alpha,\mu, n} \big\}_{n\geq1}$ fulfills all the hypotheses in Theorem \ref{thm-biorth}. This complete the proof of Proposition \ref{prop-incres-seq}.
\end{proof}
We will finish this section giving a result on the set of eigenfunctions of the operators $L$ and
$L^*$. It reads as follows:
\begin{proposition}\label{rieszbasisresult}
Let us consider the sequences
\begin{equation}\label{BB*}
\mathcal{B} = \big\{\psi_n^{(1)} , \psi_n^{(2)}, \quad n \geq 1 \big\} \quad \text{and} \quad
\mathcal{B}^* = \big\{\Psi_n^{(1)} , \Psi_n^{(2)}, \quad n \geq 1 \big\}.
\end{equation}
Then,
\begin{enumerate}
\item $\mathcal{B}$ and $\mathcal{B}^*$ are biorthogonal families in $L^2 (0, 1)^2$.
\item $\mathcal{B}$ and $\mathcal{B}^*$ are complete sequences in $L^2 (0, 1)^2$.
\item The sequences $\mathcal{B}$ and $\mathcal{B}^*$ are biorthogonal Riesz bases of $L^2 (0, 1)^2$.
\item The sequence $\mathcal{B}^*$ is a basis of $H_{\alpha,0}^{1,\mu}(0,1)^2$ and $\mathcal{B}$ is its biorthogonal basis in $H^{-1,\mu}_{\alpha}(0,1)^2$.
\end{enumerate}
\end{proposition}

\begin{proof}
From the expressions of $\psi_n^{(i)}$ and $\Psi_n^{(i)}$, we can write
$$
\psi_n^{(i)} = U_{i} \Phi_{\alpha,\mu, n} \quad \text{and}\quad \Psi_n^{(i)} = V_{i} \Phi_{\alpha,\mu, n},\quad i=1,2,\quad  n\geq 1,
$$
where $ U_{i}, V_{i} \in \mathbb{R}^2$ and $\Phi_{\alpha,\mu, n}$ is given in \eqref{eigenfunctions}.

\smallskip
\begin{enumerate}
\item It is not difficult to check that $\{U_i\}_{i=1,2}$ and $\{V_i\}_{i=1,2}$ are
biorthogonal families of $\mathbb{R}^2$. Moreover,
since $(\Phi_{\alpha,\mu, n})_{n\geq1}$ is an orthonormal basis for $L^2(0,1)$, we readily deduce
$$
\langle \psi_n^{(i)},\Psi_k^{(j)} \rangle =  (U_{i})^{tr} V_{j} \langle \Phi_{\alpha,\mu, n},\Phi_{\alpha,\mu, k} \rangle = \delta_{ij} \delta_{nk}, \quad \forall n,k\geq 1,\quad i,j=1, 2.
$$
This proves the claim.
\item We will use \cite[Lemma 1.44]{chris}. For this purpose, let us consider $f= (f_1, f_2)^{tr} \in L^2(0,1)^2$ such that
$$
\langle f  , \psi_n^{(i)} \rangle = 0,\quad \forall n\geq 1, \quad i=1, 2.
$$
If
we denote $f_{i,n}$ $(i = 1,2)$ the corresponding Fourier coefficients of the function $f_i\in L^2(0,1)$ with respect to the basis $(\Phi_{\alpha,\mu, n})_{n\geq1}$, then the previous equality can be written as
$$
(f_{1,n} , f_{2,n} ) [U_{1} | U_{2}]= 0_{\mathbb{R}^{2}}, \quad \forall n\geq 1.
$$
Using the fact that $det [U_{1} | U_{2}] \neq 0$, we deduce $f_{1,n} = f_{2,n} = 0$, for all $n\geq 1$. This implies that $f_1=f_2=0$ (since $(\Phi_{\alpha,\mu, n})_{n\geq1}$ is an orthonormal basis in $L^2(0,1)$) and, therefore,
$f = 0$ which proves the completeness of $\mathcal{B}$.
A similar argument can be used for $\mathcal{B}^*$ and the conclusion follows immediately.
\item By \cite[Theorem 7.13]{chris}, we know that $\big\{\psi_n^{(1)} , \psi_n^{(2)}\big\}_{n \geq 1}$ is a Riesz basis for $L^2(0,1)^2$ if and only if
$\big\{\psi_n^{(1)} , \psi_n^{(2)}\big\}_{n \geq 1}$ is a complete Bessel sequence and possesses a biorthogonal system
that is also a complete Bessel sequence. Using the previous properties $1)$ and $2)$, we only have to prove that the sequence $\big\{\psi_n^{(1)} , \psi_n^{(2)}\big\}_{n \geq 1}$ and $\big\{\Psi_n^{(1)} , \Psi_n^{(2)}\big\}_{n \geq 1}$ are Bessel sequences. This amounts to prove that the series
\begin{align*}
&S_1(f)= \sum_{n\geq 1} \big[ \langle f  , \psi_n^{(1)} \rangle^2 + \langle f  , \psi_n^{(2)} \rangle^2 \big] \;
\text{and}\\
&S_2(f)= \sum_{n\geq 1} \big[ \langle f  , \Psi_n^{(1)} \rangle^2 + \langle f  , \Psi_n^{(2)} \rangle^2 \big]
\end{align*}
converge for any $f= (f_1, f_2)^{tr} \in L^2(0,1)^2$.

From the definition of the functions $\psi_n^{(i)}$ and $\Psi_n^{(i)}$, it is easy to see that there exists some constant $C>0$ such that
\begin{align*}
&S_1(f) \leq C \sum_{n\geq 1} \big( |f_{1,n}|^2 + |f_{2,n}|^2 \big) \quad
\text{and}\\
&S_2(f) \leq C \sum_{n\geq 1} \big( |f_{1,n}|^2 + |f_{2,n}|^2 \big).
\end{align*}
Recall that $f_{i,n}$ is the Fourier coefficient of the function $f_{i}\in L^2(0,1)$ ($i=1,2$) with respect to $\Phi_{\alpha,\mu, n}$. Accordingly, the series $S_1(f)$ and $S_2(f)$ converge since $(\Phi_{\alpha,\mu, n})_{n\geq1}$ is an orthonormal basis for $L^2(0,1)$. We obtain thus the proof of desired result.
\item For showing item $4)$ we make use of \cite[Theorem 5.12]{chris}.
First, by taking $L^2(0,1)$ as a pivot space, one has
$$
H_{\alpha,0}^{1,\mu}(0,1) \subset  L^2(0,1)  \subset \big( H_{\alpha,0}^{1,\mu}(0,1) \big)^{'} = H_\alpha^{-1,\mu}(0,1).
$$
Furthermore, observe that $\mathcal{B}^* \subset H_{\alpha,0}^{1,\mu}(0,1)^2 $ and is complete in this space since it is in $L^2(0,1)^2$.
On the other hand, by the definition of the duality pairing, we have
$$\langle \psi_n^{(i)},\Psi_k^{(j)}\rangle_{H^{-1,\mu}_\alpha, H_{\alpha,0}^{1,\mu}} = \langle \psi_n^{(i)},\Psi_k^{(j)} \rangle=\delta_{ij} \delta_{nk},\quad \forall n,k\geq 1, \quad i,j=1,2.
$$
Thus, $\mathcal{B} \subset H_\alpha^{-1,\mu}(0,1)^2 $ and is biorthogonal to $\mathcal{B}^*$, which also yields that $\mathcal{B}^*$ is minimal in $H_{\alpha,0}^{1,\mu}(0,1)^2$ thanks to \cite[Lemma 5.4]{chris}.
To conclude the proof, it remains to prove that for any $f=(f_1,f_2)\in H_{\alpha,0}^{1,\mu}(0,1)^2$, the series
$$
S(f)= \sum_{n\geq 1} \big[ \langle \psi_n^{(1)}, f \rangle_{H^{-1,\mu}_\alpha, H_{\alpha,0}^{1,\mu}} \Psi_n^{(1)} + \langle  \psi_n^{(2)}, f \rangle_{H^{-1,\mu}_\alpha, H_{\alpha,0}^{1,\mu}} \Psi_n^{(2)}\big]
$$
converges in $H_{\alpha,0}^{1,\mu}(0,1)^2 $.

Using again the definitions of $ \psi_n^{(i)}$ and $\Psi_n^{(i)}$, one can prove that
\begin{equation*}
\langle \psi_n^{(1)}, f \rangle_{H^{-1,\mu}_\alpha, H_{\alpha,0}^{1,\mu}} \Psi_n^{(1)} = \frac{a_1}{\alpha_2^2 + a_1} \begin{pmatrix} \frac{\alpha_2^2}{a_1} f_{1,n}  - \frac{\alpha_2}{a_1} f_{2,n}  \\ - \alpha_2 f_{1,n} + f_{2,n} \end{pmatrix}  \Phi_{\alpha,\mu, n}
\end{equation*}
and
\begin{equation*}
\langle  \psi_n^{(2)}, f \rangle_{H^{-1,\mu}_\alpha, H_{\alpha,0}^{1,\mu}} \Psi_n^{(2)} = \frac{a_1}{\alpha_1^2 + a_1} \begin{pmatrix} \frac{\alpha_1^2}{a_1} f_{1,n}  - \frac{\alpha_1}{a_1} f_{2,n}  \\ - \alpha_1 f_{1,n} + f_{2,n} \end{pmatrix}  \Phi_{\alpha,\mu, n}
\end{equation*}
where $f_{i,n}$ is the Fourier coefficient of the function $f_i \in H_{\alpha,0}^{1,\mu}(0,1)$, $i= 1,2$.

But, we know that the series $ \sum\limits_{n\geq 1} f_{i,n} \Phi_{\alpha,\mu, n}$, $i= 1,2$ converges in $H_{\alpha,0}^{1,\mu}(0,1)$ since $(\Phi_{\alpha,\mu, n})_{n\geq1} $ is an orthogonal basis for $H_{\alpha,0}^{1,\mu}(0,1)$ and $f_1, f_2 \in H_{\alpha,0}^{1,\mu}(0,1)$.
This implies that, the series
\begin{equation*}
\sum\limits_{n\geq 1} \langle \psi_n^{(1)}, f \rangle_{H^{-1,\mu}_\alpha, H_{\alpha,0}^{1,\mu}} \Psi_n^{(1)}\;\text{and}\; \sum\limits_{n\geq 1} \langle  \psi_n^{(2)}, f \rangle_{H^{-1,\mu}_\alpha, H_{\alpha,0}^{1,\mu}} \Psi_n^{(2)}
\end{equation*}
converge in $H_{\alpha,0}^{1,\mu}(0,1)^2$ and assure the convergence of $S(f)$ in $H_{\alpha,0}^{1,\mu}(0,1)^2$. This concludes the proof of the result.
\end{enumerate}
\end{proof}


\section{Boundary approximate controllability}\label{Section-approx}
We will devote this section to proving the approximate controllability at time $T > 0$ of
system \eqref{bound-sys}. To this aim, we are going to use Theorem \ref{thm-biorth0}. 

First of all, using Lemma \ref{gapresult} and similar techniques as in Proposition \ref{prop-incres-seq}, one can prove that the following result holds.
\begin{proposition}\label{prop-incres-seq0}
Assume that condition \eqref{spectre-cond} holds. Then, the family defined in \eqref{increas-sequence}
satisfies \eqref{Cv-gap-orthog0}.
\end{proposition}

Now, we are ready to state our first main result on approximate controllability. One has:
\begin{theorem} \label{Thm-approx}
Let  $\mu\leq \mu(\alpha)$ and consider $\alpha_1$ and $\alpha_2$ the eigenvalues of the matrix $A$. Then, system \eqref{bound-sys} is approximately controllable in $H^{-1,\mu}_{\alpha}(0,1)^2$ at time $T >0$ if and only if $\alpha_1$ and $\alpha_2$ satisfy condition \eqref{spectre-cond}.
\end{theorem}

\begin{proof}
As said in section \ref{Section-prel}, in order to prove this theorem we will follow  a duality approach leading us to
study the unique continuation property for the adjoint system.

\smallskip
\noindent\textbf{Necessary condition:}
By contradiction, let us assume that condition \eqref{spectre-cond} does not hold, i.e., that there is $n_0, l_0 \in \mathbb{N}^*$ with $n_0\neq l_0$ such that
\begin{equation*}
\lambda_{\alpha,\mu ,n_0}^{(1)} = \lambda_{\alpha,\mu,l_0}^{(2)}:= \lambda.
\end{equation*}
Let us see that the unique continuation property for
the adjoint system \eqref{adjoint-sys} is no longer valid. Indeed, let us take
$\varphi_0= a \Psi_{n_0}^{(1)} + b \Psi_{l_0}^{(2)} \in H_{\alpha,0}^{1,\mu}(0,1)^2$, with $a, b \in \mathbb{R}$ to be determined. In this case, it is not difficult to see that the corresponding solution to the adjoint problem \eqref{adjoint-sys} is
given by
$$
\varphi(t,x) = (a  \Psi_{n_0}^{(1)}(x) + b \Psi_{l_0}^{(2)}(x) ) e^{-\lambda (T-t)},\quad \forall (t,x)\in Q.
$$
By recalling the definition of $\Psi_{n}^{(i)}$, we have that
\begin{align*}
B^* (x^\alpha\varphi_x) (t, 1)
&= B^* \big(a  \Psi_{n_0,x}^{(1)}(1) + b \Psi_{l_0,x}^{(2)}(1) \big) e^{-\lambda (T-t)} \\
& =  \Big( a B^* V_{1} (x^\alpha(\Phi_{\alpha,\mu, n_0})_x)(1) + b B^*  V_{2}  (x^\alpha(\Phi_{\alpha,\mu, l_0})_x)(1) \Big)e^{-\lambda (T-t)}\\
& =  - \Big(a  \frac{\alpha_2}{a_1}   (x^\alpha(\Phi_{\alpha,\mu, n_0})_x)(1)  + b \frac{\alpha_1}{\alpha_1^2 + a_1} (x^\alpha(\Phi_{\alpha,\mu, l_0})_x)(1) \Big) e^{-\lambda (T-t)}.
\end{align*}
Then, choosing
\begin{align*}
&a =\frac{\alpha_1}{\alpha_1^2 + a_1} (x^\alpha(\Phi_{\alpha,\mu, n_0})_x)(1) \quad \text{and} \\
&b = -\frac{\alpha_2}{a_1}   (x^\alpha(\Phi_{\alpha,\mu, l_0})_x)(1),
\end{align*}
we get that $B^*(x^\alpha\varphi_x)(t, 1)=0$ but $\varphi_0\neq 0$, which proves that the unique continuation property for the adjoint system \eqref{adjoint-sys} fails to be true. This ends the proof of the necessary part.

\smallskip
\noindent\textbf{Sufficient condition:} Let us now assume that the condition \eqref{spectre-cond} holds and
prove that the unique continuation property for the solutions of the adjoint system \eqref{adjoint-sys} holds.

Let us consider $\varphi_0 \in  H_{\alpha,0}^{1,\mu}(0,1)^2$ and suppose that the corresponding solution $\varphi$ of the adjoint
problem \eqref{adjoint-sys} satisfies
\begin{equation}\label{hyp-null-observ}
B^*(x^\alpha \varphi_x) (t, 1)  = 0, \quad \forall t\in (0,T).
\end{equation}
From Proposition \ref{rieszbasisresult}, we know that $\mathcal{B}^*$ is a basis for $H_{\alpha,0}^{1,\mu}(0,1)^2$ and thus $\varphi_0\in H_{\alpha,0}^{1,\mu}(0,1)^2$ can be written as
\begin{equation*}
\varphi_0 = \sum_{n\geq 1} ( b_n \Psi_n^{(1)} + c_n \Psi_n^{(2)} ),
\end{equation*}
where
\begin{equation*}
b_n = \langle \psi_n^{(1)}, \varphi_0 \rangle_{H^{-1,\mu}_\alpha,H_{\alpha,0}^{1,\mu}} \quad \text{and}\quad c_n = \langle \psi_n^{(2)}, \varphi_0 \rangle_{H^{-1,\mu}_\alpha ,H_{\alpha,0}^{1,\mu}},\quad \text{for any}\quad n\geq 1.
\end{equation*}
Using Proposition \ref{propo-eigenf-LL*}, the corresponding solution $\varphi$ of system \eqref{adjoint-sys} associated to $\varphi_0$ is given by
\begin{equation*}
\varphi(t,\cdot)= \sum_{n\geq 1} \Big(b_n \Psi_{n}^{(1)}e^{-\lambda_{\alpha,\mu,n}^{(1)}(T-t)}
 +
c_n \Psi_{n}^{(2)} e^{-\lambda_{\alpha,\mu,n}^{(2)}(T-t)}\Big),\quad \forall t\in(0,T).
\end{equation*}
On the other hand, direct computations show that
\begin{equation*}
\big(x^\alpha(\Phi_{\alpha,\mu, n})_x\big)(1) =  \frac{(2-\alpha)^{\frac{3}{2}} j_{\nu(\alpha,\mu),n} }{2|J'_{\nu(\alpha,\mu)}(j_{\nu(\alpha,\mu) , n})|} J'_{\nu(\alpha,\mu)}(j_{\nu(\alpha,\mu), n} ) .
\end{equation*}
Hence,
\begin{equation}\label{Ineqappro1}
\begin{aligned}
&0=B^*(x^\alpha \varphi_x) (T-t, 1)\notag\\
&= \sum_{n\geq 1}  B^* \big( b_n  (x^\alpha \Psi_{n,x}^{(1)})(1) e^{-\lambda_{\alpha,\mu, n}^{(1)} t} + c_n  (x^\alpha \Psi_{n,x}^{(2)})(1) e^{-\lambda_{\alpha,\mu, n}^{(2)} t} \big)\notag\\
&= \big(x^\alpha(\Phi_{\alpha,\mu, n})_x\big)(1) \big( b_n B^* V_{1} e^{-\lambda_{\alpha,\mu, n}^{(1)} t} + c_n B^* V_{2} e^{-\lambda_{\alpha,\mu, n}^{(2)} t} \big)\notag\\
&= - \dfrac{(2-\alpha)^\frac{3}{2}}{2}  \sum_{n\geq 1} \dfrac{J'_{\nu(\alpha,\mu)} (j_{\nu(\alpha,\mu), n} ) }{|J'_{\nu(\alpha,\mu)}(j_{\nu(\alpha,\mu), n})|} \frac{j_{\nu_{{(\alpha,\mu)}, n}}}{a_1} \big( b_n \alpha_2 e^{-\lambda_{\alpha,\mu, n}^{(1)} t} + c_n \alpha_1 \frac{a_1}{\alpha_1^2 + a_1} e^{-\lambda_{\alpha,\mu, n}^{(2)} t} \big) \\
&= - \dfrac{(2-\alpha)^\frac{3}{2}}{2}   \sum_{n\geq 1} \dfrac{J'_{\nu(\alpha,\mu)} (j_{\nu(\alpha,\mu), n})}{|J'_{\nu(\alpha,\mu)}(j_{\nu(\alpha,\mu), n})|} \frac{j_{\nu(\alpha,\mu), n}}{a_1} e^{\alpha_2 t} \big( b_n \alpha_2 e^{-(\lambda_{\alpha,\mu, n}^{(1)}+\alpha_2)t} +  c_n \alpha_1 \frac{a_1}{\alpha_1^2 + a_1} e^{-(\lambda_{\alpha,\mu, n}^{(2)}+\alpha_2) t} \big).
\end{aligned}
\end{equation}
From Proposition \ref{prop-incres-seq0}, we can apply Theorem \ref{thm-biorth0} in order to deduce the existence of a biorthogonal
family $\{q_n^{(1)}, q_n^{(2)} \}_{n\geq1}$ to
$\{e^{-(\lambda_{\alpha,\mu, n}^{(1)}+\alpha_2) t} , e^{-(\lambda_{\alpha,\mu, n}^{(2)}+\alpha_2) t}  \}_{n\geq1}$ in $ L^2(0,T)$.
Thus, the previous identity, in particular, implies
\begin{equation*}
\left\{
\begin{array}{lll}
\int_{0}^{T} B^* (x^\alpha \varphi_x) (T-t, 1) \, e^{- \alpha_2 t} \, q_n^{(1)} (t) \, dt &=
- \dfrac{(2-\alpha)^\frac{3}{2}}{2} \dfrac{J'_{\nu(\alpha,\mu)} (j_{\nu(\alpha,\mu), n}) }{|J'_{\nu(\alpha,\mu)}(j_{\nu(\alpha,\mu), n})|} \frac{j_{\nu(\alpha,\mu), n}}{a_1} b_n \alpha_2 \\
 &=0, \quad \forall n \geq 1\\\\
\int_{0}^{T}  B^*(x^\alpha \varphi_x)  (T-t, 1) \, e^{- \alpha_2 t} \, q_n^{(2)} (t) \, dt 
&= - \dfrac{(2-\alpha)^\frac{3}{2}}{2} \dfrac{J'_{\nu(\alpha,\mu)} (j_{\nu(\alpha,\mu), n}) }{|J'_{\nu(\alpha,\mu)}(j_{\nu(\alpha,\mu), n})|} \frac{j_{\nu(\alpha,\mu), n}}{a_1} c_n \alpha_1 \frac{a_1}{\alpha_1^2 + a_1} \\
&=0, \quad \forall n \geq 1.
\end{array}
\right.
\end{equation*}
Thus $b_n = c_n = 0$ for any $n\geq 1$. In conclusion, $\varphi_0=0$. This proves the continuation property for the solutions to the adjoint problem \eqref{adjoint-sys} and, thanks to Theorem \ref{approxnullcontrol}, the
approximate controllability of system  \eqref{bound-sys} at any positive time $T$ holds.
\end{proof}

\section{Boundary null controllability}\label{Section-null}

In this section, we will address the main achievement of this work which is the boundary null controllability result of system \eqref{bound-sys}, providing an estimate of the control cost as a function of $T$. In this sense, one has:

\begin{theorem}\label{Thm-null-cont}
Let  $\mu\leq \mu(\alpha)$ and consider by $\alpha_1$ and $\alpha_2$ the eigenvalues of $A$ satisfying \eqref{spectre-cond}.
Then, for every
$T > 0$ and $y_0\in H^{-1,\mu}_{\alpha}(0,1)^2$ there exists a null control $v\in L^2(0,T)$ for system \eqref{bound-sys}
which, in addition, satisfies
\begin{equation}\label{control-cost}
\|v\|_{L^2(0,T)} \leq C e^{C T+ \frac{C}{T}} \|y_0 \|_{H^{-1,\mu}_{\alpha}}.
\end{equation}
\end{theorem}

\begin{proof}
\smallskip
To prove Theorem \ref{Thm-null-cont}, we transform the controllability
problem into a moment problem. Using Proposition \ref{mainrelation}, we deduce that the control $v\in L^2(0,T)$ drives the solution of \eqref{bound-sys} to zero at time $T$ if and only if $v \in L^2(0,T)$ fulfills
\begin{equation}\label{firststep}
 \int_0^T B^*(x^\alpha \varphi_x) (t,1) \, v(t) \; dt = \langle y_0 , \varphi(0, \cdot) \rangle_{H^{-1,\mu}_{\alpha}, H_{\alpha,0}^{1,\mu}}, \quad \forall \varphi_0 \in H_{\alpha,0}^{1,\mu}(0, 1)^2
\end{equation}
where $\varphi \in  C^0\big([0, T];  H_0^{1,\mu}(0,1)^2 \big) \cap L^2\big(0, T; H^{2,\mu}(0,1)^2 \cap H^{1,\mu}_{0}(0,1)^2 \big)$ is the solution of the adjoint system \eqref{adjoint-sys} associated to $\varphi_0$.

Using Proposition \ref{propo-eigenf-LL*}, the corresponding solution $\varphi$ of system \eqref{adjoint-sys} associated to $\varphi_0$ is given by
\begin{align*}
&\varphi(t,x)=
&\sum_{k\geq 1} \Big(\langle \psi_k^{(1)}, \varphi_0 \rangle_{H^{-1,\mu}_{\alpha},H_{\alpha,0}^{1,\mu}}  \Psi_{k}^{(1)}e^{-\lambda_{\mu,k}^{(1)}(T-t)}
+  
\langle \psi_k^{(2)}, \varphi_0 \rangle_{H^{-1,\mu}_{\alpha},H_{\alpha,0}^{1,\mu}} \Psi_{k}^{(2)} e^{-\lambda_{\mu,k}^{(2)}(T-t)}\Big).
\end{align*}
From Proposition \ref{rieszbasisresult}, we have that $\mathcal{B}^*$ is a basis for $H_{\alpha,0}^{1,\mu}(0,1)^2$. In particular, we also deduce that
$\varphi(t,x) = \Psi_n^{(i)}(x) e^{-\lambda_{\alpha,\mu, n}^{(i)} (T-t)}$ is the solution of system \eqref{adjoint-sys} corresponding to
$\varphi_0= \Psi_n^{(i)} \in H_{\alpha,0}^{1,\mu}(0,1)^2$.
Therefore, we can deduce that the identity \eqref{firststep} is equivalent to
$$
\int_0^T B^* (x^\alpha \Psi_{n,x}^{(i)}) (1) v(t) e^{-\lambda_{\alpha,\mu, n}^{(i)} (T-t)} dt =  e^{-\lambda_{\alpha,\mu, n}^{(i)} T} \langle y_0 , \Psi_n^{(i)} \rangle_{H^{-1,\mu}_{\alpha}, H_{\alpha,0}^{1,\mu}}, \; \forall n\geq 1, \; i=1,2.
$$
Taking
into account the expressions of $\Psi_n^{(i)}$ (see \eqref{eigenf-L*}), we infer that $v\in L^2(0,T)$ is
a null control for system \eqref{bound-sys} associated to $y_0$ if and only if
\begin{align*}
\frac{(2-\alpha)^{\frac{3}{2}}  j_{\nu(\alpha,\mu), n} }{2|J'_{\nu(\alpha,\mu)}(j_{\nu(\alpha,\mu), n})|}  J'_{\nu(\alpha,\mu)} (j_{\nu(\alpha,\mu), n} ) B^* V_{i} &\int_0^T v(t) e^{-\lambda_{\alpha,\mu, n}^{(i)} (T-t)} dt \\
&\quad =  e^{-\lambda_{\alpha,\mu, n}^{(i)} T} \langle y_0 , \Psi_n^{(i)} \rangle_{H^{-1,\mu}_{\alpha}, H_{\alpha,0}^{1,\mu}},\; \forall n\geq 1, \; i=1,2
\end{align*}
and equivalently,
\begin{equation}\label{moment}
\int_0^T v(t) e^{-\lambda_{\alpha,\mu, n}^{(i)} (T-t)} dt = C_{\nu(\alpha,\mu),n}^{(i)} , \; \forall n\geq 1, \; i=1,2,
\end{equation}
where
$C_{\nu(\alpha,\mu),n}^{(i)}$ is given by
\begin{equation*}
C_{\nu(\alpha,\mu),n}^{(i)}=  \frac{2 |J'_{\nu(\alpha,\mu)}(j_{\nu(\alpha,\mu), n})| e^{-\lambda_{\mu, n}^{(i)} T} }{ (2-\alpha)^{\frac{3}{2}}  j_{\nu(\alpha,\mu), n}  J'_{\nu(\alpha,\mu)}(j_{\nu(\alpha,\mu), n} ) B^* V_{i}}  \langle y_0 , \Psi_n^{(i)}\rangle_{H^{-1,\mu}_{\alpha}, H_{\alpha,0}^{1,\mu}}, \;  \forall n\geq 1, \; i=1,2.
\end{equation*}
Performing the change of variable $s= T/2- t$ in \eqref{moment}, the controllability problem reduces then to the following moment problem:
Given $y_0 \in  H^{-1,\mu}_{\alpha}(0,1)^2$ find $v\in L^2(0,T)$ such that $u(s) = v(T/2- s) e^{\alpha_2s} \in L^2(-T/2,T/2)$ satisfies
\begin{equation}\label{moment1}
\int_{-T/2}^{T/2} u(s) e^{-(\lambda_{\alpha,\mu, n}^{(i)}+\alpha_2)s} ds = \widehat{C}_{\nu(\alpha,\mu),n}^{(i)} , \quad \forall n\geq 1, \quad i=1,2,
\end{equation}
with
\begin{equation}\label{expres1}
\widehat{C}_{\nu(\alpha,\mu),n}^{(i)} = e^{\lambda_{\alpha,\mu, n}^{(i)} T/2} C_{\nu(\alpha,\mu),n}^{(i)}.
\end{equation}
At this stage, the strategy to solve the moment problem \eqref{moment1} is to use the concept of biorthogonal family. In fact, Proposition \ref{prop-incres-seq} and Theorem \ref{thm-biorth} guarantee the existence of $T_0 > 0$, such that for any $T\in (0,T_0)$, there exists a biorthogonal family $\{q_n^{(1)}, q_n^{(2)} \}_{n\geq1}$ to
$\{e^{-(\lambda_{\alpha,\mu, n}^{(1)}+\alpha_2) t} , e^{-(\lambda_{\alpha,\mu, n}^{(2)}+\alpha_2) t}  \}_{n\geq1}$ in $ L^2(-T/2,T/2)$ which also satisfies
\begin{equation}\label{boundqni2}
\|q_n^{(i)} \|_{L^2(-T/2,T/2)} \leq C e^{\sqrt{\Re (\lambda_{\alpha,\mu, n}^{(i)}+\alpha_2)} + \frac{C}{T}}, \qquad \forall n \geq 1, \quad i=1,2.
\end{equation}
for some positive constant $C$ independent of $T$.

For $T<T_0$, a solution to the moment problem \eqref{moment1} is then given for every
$t\in(0,T)$ by
$$u(t) = \sum_{n\geq1} ( \widehat{C}_{\nu(\alpha,\mu),n}^{(1)} q_n^{(1)}(t)  +  \widehat{C}_{\nu(\alpha,\mu),n}^{(2)} q_n^{(2)}(t) ).$$
Thus
\begin{equation}\label{nullcontrol}
v(t)= \sum_{n\geq1}  \big( \widehat{C}_{\nu(\alpha,\mu),n}^{(1)} q_n^{(1)}(T/2- t)
+  \widehat{C}_{\nu(\alpha,\mu),n}^{(2)} q_n^{(2)}(T/2- t) \big) e^{-\alpha_2(T/2- t)}.
\end{equation}
The only remaining point is to prove that $v \in L^2(0,T)$ and to estimate its norm with respect to $T$
and $y_0$. This can be achieved thanks to the estimate \eqref{boundqni2}.
Indeed, from the expression of $\Psi_n^{(i)}$ and $\lambda_{\alpha,\mu, n}^{(i)}$, we can easily deduce the existence of a constant $C_1>0$ such that for $i=1,2$ :
\begin{equation*}
\| \Psi_n^{(i)}\|_{\mu}  \leq C_1 \sqrt{\lambda_{\alpha,\mu, n}}= C_1  j_{\nu(\alpha,\mu), n},\quad \; \forall n\geq 1.
\end{equation*}

From \eqref{expres1}, it is easy to see that there exists a new constants $C$ not depending on $n$ and $T$ such that
\begin{equation}\label{estinCnu}
|\widehat{C}_{\nu(\alpha,\mu),n}^{(i)}| \leq C e^{-\lambda_{\alpha,\mu, n}^{(i)} T/2}  \|y_0 \|_{ H^{-1,\mu}_{\alpha}},\quad  \forall n\geq 1,\quad i=1,2.
\end{equation}
Coming back to the expression \eqref{nullcontrol} of the null control $v$, taking into account the definition of $\lambda_{\alpha,\mu, n}^{(i)}$ and using the estimates \eqref{boundqni2} and \eqref{estinCnu}, we get
\begin{equation}\label{step2}
\|v\|_{L^2(0,T)} \leq  C e^{CT} \|y_0 \|_{ H^{-1,\mu}_{\alpha}} \sum_{n\geq1}  e^{-\lambda_{\alpha,\mu, n} T/2} e^{C\sqrt{\lambda_{\alpha,\mu, n}} + \frac{C}{T}}.
\end{equation}
Moreover, Young's inequality gives
$$ C\sqrt{\lambda_{\alpha,\mu, n}} \leq \frac{\lambda_{\alpha,\mu, n}T}{4}+ \frac{C^2}{T}$$
for every $n\geq1$ and $T>0$, so that
\begin{equation*}
\|v\|_{L^2(0,T)} \leq  C e^{C T+ \frac{C}{T}} \|y_0 \|_{ H^{-1,\mu}_{\alpha}} \sum_{n\geq1}  e^{-\lambda_{\alpha,\mu, n} T/4}.
\end{equation*}
On the other hand, by \eqref{boundcase1} and \eqref{boundcase2}, it can be easily checked that there exists a constant  $C> 0$ such that
$$
C  n^2 \leq  \lambda_{\alpha,\mu, n} =  j_{\nu(\alpha,\mu), n}^2, \quad \forall n \geq 1.
$$
Finally, for every $T<T_0$, we then have
\begin{align*}
\|v\|_{L^2(0,T)} &\leq  C e^{C T+ \frac{C}{T}} \|y_0 \|_{ H^{-1,\mu}_{\alpha}} \sum_{n\geq1}  e^{-C  n^2 T} \\
&\leq  C e^{C T+ \frac{C}{T}} \|y_0 \|_{ H^{-1,\mu}_{\alpha}} \int_{0}^{\infty} e^{-C T s^2}\,ds \\
&= C e^{C T+ \frac{C}{T}} \|y_0 \|_{ H^{-1,\mu}_{\alpha}} \sqrt{\frac{\pi}{T}}\\
&\leq C e^{C T+ \frac{C}{T}} \|y_0 \|_{ H^{-1,\mu}_{\alpha}},
\end{align*}
where $C$ is independent of $T$. This inequality shows that
$v \in L^2(0, T)$ and yields the desired estimate on the null control in the case where $T<T_0$.
The case $T\geq T_0$ is actually reduced to the previous one. Indeed,
any continuation by zero of a control on $(0,T_0/2)$ is a control on $(0,T)$ and the estimate
follows from the decrease of the cost with respect to the time.
This completes the proof of Theorem \ref{Thm-null-cont}.

\end{proof}




\end{document}